\documentclass[a4paper,11pt]{amsart}

\usepackage{geometry}
\usepackage{amssymb,amsmath,amsthm,amscd}
\usepackage[ansinew]{inputenc} 
\usepackage[english]{babel}

\usepackage{newtxtext}
\usepackage{newtxmath}
\usepackage[T1]{fontenc}
\usepackage{mathrsfs}
\usepackage{amscd}
\usepackage{currfile}
\usepackage{bm}
\usepackage{graphicx}
\usepackage{multirow}
\usepackage[colorlinks]{hyperref}
\usepackage{color}
\usepackage{tikz}
\usepackage{floatflt}

\newtheorem{theorem}{Theorem}
\newtheorem*{theorem*}{Theorem}
\newtheorem{definition}[theorem]{Definition}

\newtheorem{proposition}[theorem]{Proposition}
\newtheorem{corollary}[theorem]{Corollary}
\newtheorem{example}[theorem]{Example}
\newtheorem{examples}[theorem]{Examples}

\theoremstyle{definition}
\newtheorem{remark}[theorem]{Remark}

\newcommand{\oo}{{\mathbb{O}}}
\newcommand{\hh}{{\mathbb{H}}}

\newcommand{\cc}{{\mathbb{C}}}
\newcommand{\rr}{{\mathbb{R}}}
\newcommand{\zz}{{\mathbb{Z}}}
\newcommand{\nn}{{\mathbb{N}}}

\newcommand{\Eu}{{\mathbb{E}}}

\newcommand{\Aa}{\mathbb{A}}
\newcommand{\s}{{\mathbb{S}}}

\newcommand{\dB}{\overline\partial_{\B}}
\newcommand{\dM}{\overline\partial_{M}}
\newcommand{\GB}{\Gamma_\B}

\newcommand{\B}{\mathcal{B}}

\newcommand{\sr}{\mathcal{SR}}

\newcommand{\SL}{\mathcal{S}}

\newcommand{\I}{\mathcal{I}}
\newcommand{\F}{\mathcal{F}}
\newcommand{\M}{\mathcal{M}}
\newcommand{\PP}{\mathcal{P}}

\newcommand{\dibar}{\overline\partial}

\newcommand{\difbar}{\overline{\vartheta}_\B}
\newcommand{\difM}{\vartheta_M}
\newcommand{\difbarM}{\overline{\vartheta}_M}
\newcommand\IM{\operatorname{Im}}
\newcommand\RE{\operatorname{Re}}

\newcommand\vs[1]{{#1}_s^\circ}

\newcommand\Span{\operatorname{span}}

\newcommand{\OO}{\Omega}

\newcommand{\mbb}{\mathbb}

\newcommand{\R}{\mbb{R}}
\newcommand{\mc}{\mathcal}

\newcommand{\dd[3]}{\frac{\partial #2}{\partial #3}}
 
 \newcommand{\C}{\mbb{C}}
 
 \newcommand{\Q}{\mc{Q}}

 \newcommand{\bc}{\begin{center}}
 \newcommand{\ec}{\end{center}}

\newcommand{\Rn}{\mathbb R_n}

\newcommand{\AM}{\mathcal{AM}}

\newcommand{\DD}{\underline{D}}
\newcommand{\x}{\underline x}
\newcommand{\sC}{\underline S_\B}
\newcommand{\SA}{\underline S_A}
\newcommand{\wsC}{\widetilde{\underline S}_\B}
\newcommand{\wSA}{\widetilde{\underline S}_A}
\newcommand{\tGamma}{\widetilde\Gamma_\B}

\newcommand\Tau{T}

\begin{document}

\title[Fueter trees for Dunkl-regular functions over alternative *-algebras]
{Fueter trees for Dunkl-regular functions\\ over alternative *-algebras
}

\author{Alessandro Perotti
}
\address{Department of Mathematics, University of Trento, Via Sommarive 14, Trento Italy}
\address{ORCID: 0000-0002-4312-9504}
\email{alessandro.perotti@unitn.it}
\thanks{The author is a member the INdAM Research group GNSAGA and was partially supported by the grant ``INdAM - GNSAGA Project, Teoria delle funzioni ipercomplesse e applicazioni
and ``INdAM - GNSAGA Project" CUP E53C25002010001''}

\begin{abstract}
We prove a general Fueter Theorem over real alternative *-algebras. We show that a suitable power of the Laplacian maps Dunkl-regular functions to Dunkl monogenic functions with axial symmetries. Using the embedding of hypercomplex function theories in the class of Dunkl monogenic functions, we subsume several Fueter-type results known in the literature and obtain the most general form for the action of the Laplacian on function spaces over hypercomplex subspaces. We show that Fueter Theorems are in a one-to-one correspondence with a class of graphs, the Fueter trees, that describe the interactions between Dunkl-regular function spaces and the relation with the iterated Laplacian. We obtain that the number of distinct Fueter trees on a hypercomplex space of dimension $n+1$ is equal to the number of partitions in odd parts of the integer $n$. 
\end{abstract}


\maketitle

\section{Introduction}

The purpose of this work is to prove a general Fueter Theorem over real alternative *-algebras. In the original statement of the Theorem \cite{Fueter1934}, Fueter showed that the Laplacian in four real variables 
provided a method to produce a large class of `regular' quaternionic functions, those defined by an analogue of the Cauchy-Riemann equations and having an axial symmetry with respect to the real axis. In the modern language of slice-regular function theory, the Laplacian maps slice-regular functions into axially monogenic functions (see e.g.\ \cite[Theorem 3.6.3]{Harmonicity}).  Fueter Theorem was later generalized to different settings: to monogenic functions over Clifford algebras by Sce \cite{Sce} and Qian \cite{Qian1997}, to octonionic regular functions by Dentoni and Sce \cite{DentoniSce}. More recently, the theorem has been extended to slice-regular functions on a hypercomplex subspace of an alternative *-algebra \cite[Theorem 27]{CRoperators} and to generalized partial-slice monogenic functions \cite{XuSabadiniFueterSce,Huo_Sabadini_Xu_arXiv26}.  In these higher dimensional extensions of Fueter Theorem the role of the Laplacian is taken by a suitable power of the Laplacian of the hypercomplex subspace.
A different Fueter-type Theorem, with the Euclidean Laplacian replaced by other differential operators of the second order, has been proved in \cite{GhiloniStoppatoFueterSce} for the class of $T$-regular functions on associative algebras.


The recent work \cite{BinosiPerotti2025} proved that all the function theories recalled above can be embedded in the family of Dunkl monogenic functions over hypercomplex subspaces. Moreover, it showed that the prescribed axial symmetries of the functions can be encoded in the Dunkl multiplicities. 
This property naturally led to the introduction of a class of function spaces, whose elements are called \emph{Dunkl-regular functions}, that refine Dunkl monogenic function theory and Dunkl harmonic analysis. 
In this work we apply results from Dunkl function theory to show (Theorem \ref{teo:DeltaM}) how the Laplacian acts between these spaces of Dunkl-regular functions. As a corollary (Corollary \ref{cor:GFT}), we obtain the general Fueter Theorem, which subsumes all the above mentioned versions of the theorem.


Dunkl operators are differential-difference operators associated with finite reflection groups. These operators, introduced in \cite{Dunkl},  appear, e.g., in harmonic analysis and in the study of multivariate special functions. We refer to \cite{DunklXu} and \cite{Rosler} for wide accounts of Dunkl theory and its applications.
Dunkl operators were introduced in Clifford analysis in 2006 \cite{Ren_et_al}. Let $\rr_n$ be the real Clifford algebra with signature $(0,n)$ with generators $e_1,\ldots,e_n$. Consider the euclidean space $\rr^n=\langle e_1,\ldots,e_n\rangle$ embedded in the paravector space $M=\langle 1,e_1,\ldots,e_n\rangle$ of $\rr_n$. Given a finite reflection group for $\R^n$ and a root system, let $T_1,\ldots,T_n$ be the corresponding Dunkl operators. The Dunkl-Dirac operator $\underline D$ on $\R_n$ is defined as $\underline D=\sum_{i=1}^ne_iT_i$. 

Besides Clifford analysis \cite{BDS,GHS}, 
another higher dimensional generalization of complex analysis, slice analysis, was developed over the last twenty years. 
The theory of slice-regular functions, that includes also standard polynomials, was introduced in the quaternionic setting \cite{GeSt2006CR,GeSt2007Adv} and then extended to Clifford algebras, octonions and to any real alternative *-algebra \cite{CoSaSt2009Israel,GeStRocky,AIM2011}. See \cite{GeStoSt2013,Struppa2015Algebras} for extended references of this function theory. 

Recently, the papers \cite{binosi2025dunklapproachsliceregular,BinosiPerotti2025} provided characterizations of sliceness and slice-regularity over any hypercomplex subspace of a real alternative *-algebra $\Aa$ based on the spherical Dunkl-Dirac operator and on the  Dunkl-Cauchy-Riemann operator $D=\dd{}{x_0}+\underline D$ under an appropriate choice of the Dunkl multiplicities. This approach allows to view both Clifford analysis and slice-regular function theory, as well as other hypercomplex function theories existing in the literature, as subcases of the general theory of Dunkl monogenic functions, namely the functions annihilated by  $D$.
Let $M$ be a hypercomplex subspace of dimension $n+1$ of $\Aa$ and let $\OO$ be an open subset of $M$. To every partition $\PP$ of the set $\{1,\ldots,n\}$, a function space $\F_\PP(\OO)$ of Dunkl monogenic functions having an axial symmetry depending on $\PP$ (called \emph{$\PP$-sliceness}) can be associated. Its elements are called \emph{$\PP$-Dunkl-regular functions} (Definition \ref{def:P-Dunkl-regular}), and the sum $\kappa$ of the Dunkl multiplicities is called the \emph{Dunkl weight} of the functions in $\F_\PP(\OO)$. 
Many of these function spaces are isomorphic: the number of non-equivalent spaces is equal to the number of partitions of the number $n$  (Theorem \ref{teo:classification2}). 
Using results from Dunkl theory, in particular a formula for the Dunkl Laplacian associated with the reflection group $\zz_2^n$, we study the action of the Euclidean Laplacian $\Delta_M$ of $M$ on Dunkl-regular functions. We prove (Theorem \ref{teo:DeltaM}) that the Laplacian of a Dunkl-regular function of weight $\kappa$ can be written as the sum of Dunkl-regular functions of weight $\kappa+1$. 
When $\PP$ is an odd partition, i.e., all its elements have odd cardinality, the theorem can be applied until the image of the iterated Laplacian belongs to the space of monogenic functions. 
We then obtain the general Fueter Theorem (Corollary \ref{cor:GFT}).

\begin{theorem*}[General Fueter Theorem]
Assume that $\PP$ is an odd partition with $\ell<n$ elements.  Let $f\in\F_\PP(\OO)$ be a Dunkl-regular function of Dunkl weight $\kappa<0$. Then $\kappa=(\ell-n)/2$ is a negative integer and the iterated Laplacian $\Delta_M^{|\kappa|}f=(\Delta_M)^{\frac{n-\ell}2}f$ is a $\PP$-slice monogenic function on $\OO$: 
\[
\dM{(\Delta_M)}^{\frac{n-\ell}2}f=0, 
\]
where $\dM$ is the Cauchy-Riemann operator of $M$.
\end{theorem*}

The general Fueter Theorem is a method to produce monogenic functions on $M$ with prescribed symmetries (the $\PP$-sliceness) starting from $\PP$-Dunkl-regular functions. The iterated application of Theorem \ref{teo:DeltaM} can be represented by a graph, the \emph{Fueter tree} of the space $\F_\PP(\OO)$. 
It is a tree of height $|\kappa|=(n-\ell)/2$, root $\F_\PP(\OO)$, nodes some other spaces $\F_{\PP'}(\OO)$ with $\PP'$ a refinement of $\PP$, and leaves all equal to the space of monogenic functions $\M(\OO)$. 

Every Fueter tree corresponds to a version of the Fueter Theorem on $M$. In particular we subsume all the already known Fueter Theorems. 
When $n-\ell>2$, Theorem \ref{teo:DeltaM} adds new information also to the classical versions of Fueter Theorem. It says which are the function spaces involved at every iteration of the Laplacian and not only at the final step $(\Delta_M)^{(n-\ell)/2}$. 
The tree with maximal height on $M$ is obtained when $\ell=1$ and $|\kappa|=(n-1)/2$. It is a unary tree with root node $\sr(\OO)$, the space of slice-regular functions. 

The number of distinct Fueter trees (up to equivalence) is equal to the number of non-equivalent roots $\F_\PP(\OO)$, with $\PP$ an odd partition. This is the number of \emph{partitions in odd parts} of $n$, usually denoted by $q(n)$. Excluding the trivial tree with root and leaf $\M(\OO)$, there are, up to equivalence, $q(n)-1$ distinct Fueter trees. Therefore there are $q(n)-1$ distinct Fueter Theorems on $M$ (see Table \ref{table}).
For example, there is only one Fueter Theorem on the quaternions, the original version of Fueter Theorem, while there are four distinct Fueter Theorems on the octonions (see \S\ref{sec:Examples}).

We describe in more detail the structure of the paper. 
Section \ref{sec:pre} is devoted to preliminaries. We recall basic definitions
about hypercomplex subspaces of real alternative *-algebras, Cauchy-Riemann operators, monogenic functions, slice and slice-regular functions. 
In Subsection \ref{sub:Dunkl-Dirac_operators} we recall the definition of Dunkl operators associated to the reflection group $\zz_2^n$ and of the Dunkl-Dirac operator $\DD_\B$ associated to a hypercomplex basis $\B$ of $M$.  Following \cite{DeBieGenestVinet}, we give the definition of the Casimir operator $\sC$ for the superalgebra $\mathfrak{osp}(1|2)$ realization provided by the operators $\x=\IM(x)$  and $\DD_\B$. Subsection \ref{sec:Dunkl-Dirac_operators} recall the criteria for sliceness and slice regularity based on Dunkl theory. 

Section \ref{sec:dunkl_regular_function_spaces} introduces Dunkl-regular function spaces by means of intermediate Dunkl-Dirac operators.
Subsection \ref{sub:Dunkl-regular_function_spaces} introduces \emph{$\PP$-Dunkl-regular functions}. For any open subset $\OO$ of $M$ and any partition $\PP$ of the set $[n]=\{1,\ldots,n\}$, where $n=\dim M-1$, the space $\F_\PP(\OO)$ is the kernel of a Dunkl-Cauchy-Riemann operator $D_\PP$.  We recall results about the dependence of the spaces $\F_\PP(\OO)$ on $\PP$ and on the Dunkl multiplicities (Theorem \ref{teo:classification1}), and the equivalence of Dunkl-regular function spaces up to reordering of the hypercomplex basis and the partition (Theorem \ref{teo:classification2}). 
Subsection \ref{sec:Pslice} collects some results about \emph{$\PP$-slice functions}, the functions having an axial symmetry depending on $\PP$. In particular, we show that the operators $\dM$ and $\Delta_M$ preserve $\PP$-sliceness. 
Subsection \ref{sec:CK} shows how to construct Dunkl-regular polynomials using the method of Cauchy-Kovalevskaya extension (Theorem \ref{teo:CK}). In particular, we obtain bases for the subspace of homogeneous Dunkl-regular polynomials of a given degree (Proposition \ref{pro:basis} and Corollary \ref{cor:basis_associative}).  

Section \ref{sec:GFT} contains the main results. Theorem \ref{teo:DeltaM} gives the decomposition of the Laplacian of a Dunkl-regular function as a sum of Dunkl-regular functions. Corollary \ref{cor:GFT} is the general Fueter Theorem and Corollary \ref{cor:polyharmonic} shows that every Dunkl-regular function is polyharmonic. These results are illustrated through some examples. The section also introduces the concept of \emph{Fueter tree}, a graph equivalent to the Fueter Theorem. Some of these trees are presented in Subsection \ref{sec:Examples} in the context of quaternionic, octonionic or Clifford algebras.

\section{Preliminaries}\label{sec:pre}

Let $\Aa$ be a finite-dimensional real alternative *-algebra with unity $1$. We assume that $1\ne0$ and embed $\rr\simeq\rr\cdot1$ in $\Aa$. 
$\Aa$ is equipped with a $\rr$-linear anti-involution $x\mapsto x^c$ such that 
$x^c=x$ for $x$ real. Let $t(x):=x+x^c\in \Aa$ be the trace of $x$ and $n(x):=xx^c\in \Aa$  the norm of $x$. 
We denote by
\[
\s_{\Aa}:=\{J\in \Aa : t(x)=0,\ n(x)=1\}
\]
the set of imaginary units of $\Aa$ compatible with the *-algebra structure and by 
\[
\Q_{\Aa}:=\cup_{J\in \s_{\Aa}}\C_J,
\]
the quadratic cone of $\Aa$ (see \cite[Definition 3]{AIM2011}), 
where $\C_J=\Span(1,J)$ is the `slice' of $\Aa$ generated by $1$ and $J$. 
Each element $x$ of $\Q_{\Aa}$ can be written as $x=\RE(x)+\IM(x)$, with $\RE(x)=\frac{x+x^c}2$, $\IM(x)=\frac{x-x^c}2=\beta J$, where $\beta=\sqrt{n(\IM(x))}\geq0$ and $J\in\s_{\Aa}$, with unique choice of $\beta\geq0$ and $J\in\s_{\Aa}$ if $x\not\in\R$.
In the following we will also write $\x$ in place of $\IM(x)$. 
We refer to 
\cite[\S2]{AIM2011} and \cite[\S1]{AlgebraSliceFunctions} for more details and examples about real alternative *-algebras and their quadratic cones.

\subsection{Hypercomplex subspaces and Cauchy-Riemann operators}\label{sec:Hypercomplex_subspaces}

We recall some concepts introduced in \cite[\S3]{CRoperators}.
A \emph{hypercomplex subspace of $\Aa$} is a real vector subspace $M$ of $\Aa$ of dimension $\dim(M)\ge2$ such that $\R \subseteq M \subseteq \Q_{\Aa}$.
For example, $\hh$ and $\oo$ are hypercomplex subspaces of $\Aa=\hh$ and $\Aa=\oo$ respectively. The space $\hh_r=\{x=x_0+ix_1+jx_2\in\hh\;|\; x_0,x_1,x_2\in\R\}$ of reduced quaternions is another hypercomplex subspace of $\hh$. In general, the space $\rr^{n+1}$ of paravectors of the real Clifford algebra $\R_n$ of signature $(0,n)$ is a hypercomplex subspace of $\Rn$. 



Lemma \cite[Lemma 1.4]{VolumeCauchy} showed that every hypercomplex subspace $M$ has a \emph{hypercomplex basis} $\B=(v_0,v_1,\ldots,v_{n})$, i.e., a real basis of $M$ with $v_0=1$, $v_1,\ldots,v_n\in\s_M:=\s_{\Aa}\cap M$ and such that for every $i\ne j$ in the set $\{1,\ldots,n\}$, the elements $v_i$ and $v_j$ anticommute (see \cite[\S3]{CRoperators}). This basis is orthonormal w.r.t.\ a scalar product of $\Aa$ such that $\|x\|^2=n(x)$ for every $x \in M$, and its elements satisfy the identity
\begin{equation}\label{eq:anticommute}
v_i(v_j a)=-v_j(v_ia)\text{\quad for every $a\in \Aa$ and $i, j\in\{1,\ldots,n\},\ i\ne j$.}  
\end{equation}

Complete $\B$ to a orthonormal basis $\B_{\Aa}=(v_0,v_1,\ldots,v_{d-1})$ of $\Aa$ and let $L:\R^d \to \Aa$ be the real vector isomorphism defined by $\B_{\Aa}$, mapping $x=(x_0,x_1,\ldots,x_{d-1})$ to $L(x)=\sum_{\ell=0}^{d-1}x_{\ell}v_{\ell}$. 

If $\OO$ is an open subset of $M$, for $i=0,1,\ldots,n$ define the differential operators $\partial_{x_i}:C^1(\OO,\Aa)\to C^0(\OO,\Aa)$ as
\[
\textstyle\partial_{x_i} f=L\circ\dd {(L^{-1}\circ f\circ L_{|\OO'})}{x_i}\circ L^{-1},
\]
where $\OO'=L^{-1}(\OO)$. 

\begin{definition}\label{def:dM}
The \emph{Cauchy-Riemann operator of $M$} is 
\[
\dM:=\dB:=\partial_{x_0}+v_1\partial_{x_1}+\cdots +v_n\partial_{x_n}
\]
where $\B$ is any hypercomplex basis of $M$. The operator is well-defined since $\dB$ does not depend on the choice of a hypercomplex basis $\B$ of $M$ (see \cite[Remark 2]{BinosiPerotti2025}). 
\end{definition}

For example, the Cauchy-Riemann-Fueter operator on $\hh$, the Cauchy-Riemann (or Fueter-Moisil \cite{DentoniSce}) operator on $\oo$ and the Cauchy-Riemann operator of Clifford analysis over $\R_n$, are the operators $\dM$, for $M=\hh,\oo,\R^{n+1}$ respectively.

\begin{definition}\label{def:monogenic}
Let $\OO\subseteq M$ be an open set. The functions $f\in C^1(\OO,\Aa)$ in the kernel of $\dM$ are called (left) \emph{monogenic functions} on $\OO$. We write $f\in\M(\OO)$.
\end{definition}

Properties of monogenic functions on a hypercomplex subspace have been studied in \cite[\S3]{GhiloniStoppato_arXiv24} and \cite{Huo_Ren_Xu_arXiv25}.

\subsection{Slice functions and slice-regular functions}\label{sec:slice}

The functions on $\Aa$ that are compatible with the slice decomposition of the quadratic cone are called \emph{slice functions} \cite{AIM2011}. 
Given $D\subseteq\cc$ invariant w.r.t.\ complex conjugation, a function $F: D \to \Aa\otimes_{\R}\C$ is a \emph{stem function} if it satisfies  $F(\overline z)=\overline{F(z)}$ for every $z\in D$, where conjugation in $\Aa\otimes_{\R}\C$ is complex conjugation of the second factor. 
Let $\OO_D$ be the \emph{axially symmetric} (or \emph{circular}) subset of $\Q_\Aa$ defined by 
\[
\OO_D=\bigcup_{J\in\s_{\Aa}}\{\alpha+J\beta\in\Aa : \alpha,\beta\in\R, \alpha+i\beta\in D\}.
\]
The stem function $F:D \to \Aa\otimes_{\R}\C$ with components $F_\emptyset,F_1:D\to\Aa$, induces a \emph{(left) slice function} $f=\I(F):\OO_D \to \Aa$: if $x=\alpha+J\beta\in \OO_D\cap \C_J$, then  
\[ f(x)=F_\emptyset(z)+JF_1(z),\text{\quad where $z=\alpha+i\beta\in D$}. 
\]
Suppose that $D$ is open.  The slice function $f=\I(F):\OO_D \to \Aa$ is \emph{(left) slice-regular} if $F$ is holomorphic on $D$ w.r.t.\ the complex structure on $\Aa\otimes_{\R}\C$ induced by the second factor. 
We will denote by $\SL^1(\OO_D)$ the real vector space of slice functions induced by stem functions of class $\mathcal C^1$ and by $\sr(\OO_D)$ the vector subspace of slice-regular functions. 
Polynomials in $x,x^c$ of the form 
$\sum_{\alpha+\beta=k}x^\alpha (x^c)^\beta a_{\alpha,\beta}$, with coefficients $a_{\alpha,\beta}\in\Aa$, are slice functions, while polynomial functions of the form $f(x)=\sum_{j=0}^d x^ja_j$ and convergent power series with right coefficients in $\Aa$ are slice-regular.
If $\Aa=\hh$ and $\OO_D\cap\rr\ne\emptyset$, this definition of slice-regularity is equivalent to the original one proposed by Gentili and Struppa in \cite{GeSt2007Adv}. 

If $M$ is a hypercomplex subspace of $\Aa$, we can consider the restriction of a slice function defined on $\OO_D$ to the subset $\OO=\OO_D\cap M$, called axially symmetric subset of $M$. 
In view of the representation formula (see e.g.~\cite[Proposition 6]{AIM2011}), the restriction of a slice function $f=\I(F)\in\SL(\OO_D)$ to $\OO$ uniquely determines $f$ and the inducing stem function $F$. We will use the same symbol $f$ to denote the restriction $f_{|\OO}$, and the symbols $\SL(\OO)$ and $\sr(\Omega)$ to denote the sets of (restricted) slice functions and of (restricted) slice-regular functions on $\OO$, where $\OO$ is an axially symmetric open subset of $M$.

We can characterize sliceness of functions through spherical values.
To any function $f:\OO\to \Aa$, not necessarily slice, we associate its \emph{spherical value} $\vs f:\OO \to \Aa$, 
defined as
\begin{equation}\label{eq:sphericalfunctions}
\vs f(x):=\tfrac{1}{2}(f(x)+f(x^c)).
\end{equation}
The following equality holds
\begin{equation}\label{eq:spherical}
f=
\vs f+\x ^{-1}\vs{(\x f)},\text{\quad where $\x=\IM(x)$.}
\end{equation}


\begin{proposition}\cite[Proposition 9]{BinosiPerotti2025}
If $f$ is continuous on $\OO\subseteq M$, then $f$ is a slice function on $\OO$ if and only if it is a slice function on $\OO\setminus\rr$. 
Therefore $f\in\SL(\OO)\cap C(\OO,\Aa)$ if and only if the functions $\vs f$ and $\vs{(\x f)}$ are constant on every sphere $\s_x\cap M=\alpha+\beta\,\s_M$, with $x=\alpha+I\beta\in\OO\setminus\rr$. 
\end{proposition}

If $\dim M=2$, then any function $f:\OO\to\Aa$ is a slice function, since $\alpha+\beta \s_M=\{x,x^c\}$ and $\vs f$, $\vs{(\x f)}$ are invariant w.r.t.\ conjugation $x\mapsto x^c$. 
We refer the reader to \cite[\S3,4]{AIM2011} for more properties of slice functions and slice-regularity.

\subsubsection*{The global operators $\difM$, $\difbarM$ and the spherical Dirac operator}

Let $\OO$ be an open subset of the hypercomplex subspace $M$ of $\Aa$. 
We recall from \cite{Gh_Pe_GlobDiff} and \cite{CRoperators} the definition of the global differential operator $\difbar: C^1(\OO\setminus \R,A) \to C^0(\OO\setminus \R, A)$:
\begin{equation}\label{eq:theta}
\difbar=\partial_{x_0}-\x ^{-1}\Eu,
\end{equation}
where $\Eu=\sum_{i=1}^nx_i\partial_{x_i}$ is the Euler operator for $M\cap\ker(t)=\Span(v_1,\ldots,v_n)$. 
The operator $\Eu$ does not depend on the choice of $\B$, so we can define the operator $\difbarM:=\difbar$ for any choice of the hypercomplex basis $\B$ of $M$. 
If $f$ is a slice function on $\OO$, $f$ is slice-regular if and only if $\difbarM f=0$ on $\OO\setminus\R$ (see \cite[Theorem 2.2]{Gh_Pe_GlobDiff}). 

For any $i,j$ with $1\le i,j\le n$, let $L_{ij}=x_i\partial_{x_j}-x_j\partial_{x_i}$ and let
\[
\GB=-\sum_{1\le i<j\le n} v_i(v_jL_{ij})=-\tfrac12\sum_{i,j=1}^n v_i(v_jL_{ij})
\] 
be the \emph{spherical Dirac operator} associated to the basis $\B$. 
The operators $L_{ij}$ are tangential differential operators for the spheres $\s_x\cap M=\alpha+\beta\s_M$, with $x=\alpha+I\beta\in\Q_{\Aa}\setminus\rr$. 
We recall the main relation linking the operators $\dM$, $\difbarM$ and $\Gamma_\B$, that implies the invariance of $\Gamma_\B$ w.r.t.\ $\B$. 

\begin{theorem}\cite[Theorem 12]{BinosiPerotti2025}\label{teo:difference}
Let $f:\OO\to \Aa$ be a $\mathcal{C}^1$ function. 
Then it holds in $\OO\setminus\R$:
\[
\dM f-\difbarM f=-\x ^{-1} \Gamma_\B f.
\]
\end{theorem}

We can then define the \emph{spherical Dirac operator for $M$} as $\Gamma_M:=\Gamma_\B$ for any choice of the hypercomplex basis $\B$ of $M$. 



\subsection{Dunkl-Dirac operators}\label{sub:Dunkl-Dirac_operators}

We recall from \cite{DeBieGenestVinet} and \cite{Dunkl} the definition of the Dunkl operators for the abelian reflection group $\zz_2^n$. If $r_i$ is the reflection operator w.r.t.\ the $i$-th coordinate of $\rr^n$, and $k_1,\ldots,k_n\in\rr$ are the multiplicities, the \emph{Dunkl operators} $T_1,\ldots,T_n$ on $\rr^n$ associated to $\zz_2^n$ are defined as follows:
\[
T_i=\dd{}{x_i}+\frac{k_i}{x_i}(1-r_i),\quad i=1,\ldots,n,
\]
where $1$ denotes the identity operator. The operators $T_i$ commute each other and define the \emph{Dunkl-Laplace} operator $\Delta_D=\sum_{i=1}^nT_i^2$. 
In 2006 \cite{Ren_et_al} this definition was extended to the Clifford algebra $\R_n$ generated by $\{e_0=1,e_1,\ldots,e_n\}$ , introducing the Dunkl-Dirac operator $\DD=\sum_{i=1}^ne_iT_i$.  
We extend this definition to any hypercomplex subspace.

\begin{definition}\label{def:DiracDunkl}
Given a hypercomplex basis $\B=(1,v_1,\ldots,v_n)$ of the hypercomplex subspace $M$ of $\Aa$, with  coordinates $x_0,x_1,\ldots,x_n$, and an open set $\OO\subseteq M$, invariant w.r.t.\ reflections $r_i$, $i=1,\ldots,n$, the \emph{$\zz_2^n$ Dunkl-Dirac operator $\DD_\B:C^1(\OO,\Aa)\to C^0(\OO,\Aa)$} w.r.t.\ $\B$ 
is defined as
\[
\DD_\B:=\sum_{i=1}^nv_i T_{\B,i},\text{\quad where\quad} T_{\B,i}f=
\partial_{x_i}f+\frac{k_i}{x_i}(f-r_i f).
\]
The \emph{$\zz_2^n$ Dunkl-Cauchy-Riemann operator} w.r.t.\ $\B$ is defined as $D_\B:=\partial_{x_0}+\DD_\B$.
\end{definition}

The operators $T_{\B,i}$ commute each other and $D_\B$, $\DD_\B$ map $C^\ell(\OO,\Aa)$ to $C^{\ell-1}(\OO,\Aa)$. When ${\bf k}=(k_1,\ldots,k_n)=0$, then $D_\B=\dM$. 

\begin{definition}\label{def:Dunklmonogenic}
Given an open set $\OO\subseteq M$, invariant w.r.t.\ reflections $r_i$, functions $f\in C^1(\OO,\Aa)$ with $D_\B f=0$ are called (left) \emph{Dunkl monogenic functions} on $\OO$. 
\end{definition}

Let $\Delta_\B=\sum_{i=0}^n\partial_{x_i}(\partial_{x_i}\cdot)$ be the Laplacian operator induced by $\B$. 
Proposition 5 in \cite{CRoperators} proved that $\Delta_\B f=\partial_\B(\dB f)=\dB(\partial_\B f)$ for every function $f$ of class $C^2(\OO,\Aa)$, where $\partial_{\B}:=\partial_{x_0}-v_1\partial_{x_1}-\cdots -v_n\partial_{x_n}$.  Also $\partial_{\B}$  and $\Delta_\B$ do not depend on the choice of $\B$. We can then define the \emph{conjugated Cauchy-Riemann operator of $M$} and the \emph{Laplacian operator of $M$} as
\[
\partial_M:=\partial_{\B}\text{\quad  and\quad}\Delta_M:=\Delta_\B=\partial_M\dM=\dM\partial_M.
\]

 Observe that every monogenic function $f\in C^2(\OO,\Aa)$ is \emph{harmonic}, i.e., in the kernel of $\Delta_M$. In particular, every monogenic function is real analytic (see \cite[Remark 7]{BinosiPerotti2025}).

Let $D^c_\B=\partial_{x_0}-\DD_\B$ and let $\Delta_{D,M}:=D_\B D^c_\B=D^c_\B D_\B =\partial_{x_0}^2+\sum_{i=1}^n T^2_{\B,i}$ be the  \emph{Dunkl-Laplace operator} of $M$. 
Every Dunkl monogenic function of class $C^2$ is in the kernel of $\Delta_{D,M}$, i.e., \emph{Dunkl harmonic}. 
The operator $\Delta_{D,M}$ has the following form (see e.g.\ \cite[\S7.5.1]{DunklXu}):
\begin{equation}\label{eq:DeltaDM}
\Delta_{D,M}=\Delta_M+\sum_{i=1}^nk_i\left(\frac2{x_i}\partial_{x_i}-\frac{1-r_i}{x_i^2}\right).
\end{equation}


\begin{remark}
When $M$ is the paravector subspace of the Clifford algebra $\R_n$, the Dunkl-Dirac operator together with the left multiplication operator by $\x=\IM(x)$ provide a realization of the orthosymplectic Lie superalgebra $\mathfrak{osp}(1|2)$ (see \cite[Lemma 4.1]{Orsted} and \cite[Theorem 1]{deBieOrstedSombergSoucek}). 
It can proved, following the Clifford algebra case, that on any hypercomplex subspace, the operators $\DD_\B$ and $\x$ generate a superalgebra $\mathfrak{osp}(1|2)$ (see \cite[Proposition 17]{BinosiPerotti2025}). 

\end{remark}




Let  $[\ ,\ ]$ denote the commutator and $\{\ ,\ \}$ the anticommutator of a pair of real-linear operators w.r.t.\ composition. 
Following \cite{DeBieGenestVinet}, we define the \emph{Casimir operator} for the $\mathfrak{osp}(1|2)$ realization provided by $\DD_\B$:
\begin{equation}
\sC=\tfrac12\left(\left[\x,\DD_\B\right]-1\right)
\end{equation}
and the \emph{spherical Dunkl-Dirac operator} $\tGamma$ w.r.t.\ $\B$, defined as $\tGamma=\sC r$, 
where $r=\prod_{i=1}^nr_i$ is the composition of reflections $r_i$. 

\begin{proposition}\cite[Proposition 18]{BinosiPerotti2025}\label{pro:operators}
\textrm{(i)}\ The operators $\sC$ and $\tGamma$ satisfy the following commutation relations:
\begin{equation}\label{eq:commutation}
  \left\{\sC,\DD_\B\right\}=0,\ \left\{\sC,\x\right\}=0,\ \left[\tGamma,\DD_\B\right]=0,\ \left[\tGamma,\x\right]=0.
\end{equation}
Moreover, it holds
\begin{equation}\label{eq:r}
  \{\DD_\B,r\}=\{\sC,r\}=\{\x,r\}=[\Eu,r]=0.
\end{equation}
\textrm{(ii)}\ If the \emph{Dunkl weight} $\kappa=\sum_{i=1}^nk_i$ is equal to $(1-n)/2$, then the operator $\sC$ can be written as
\begin{equation}\label{casimir}
\sC=\x\,\DD_\B+\Eu.  
\end{equation}
\end{proposition}

\subsection{Dunkl-Dirac operators, sliceness and regularity}\label{sec:Dunkl-Dirac_operators}

For any hypercomplex basis $\B=(1,v_1,\ldots,v_n)$ of $M$, with coordinates $x_0,\ldots,x_n$, every homogeneous polynomial with coefficients in $\Aa$ in the $n+1$ real variables $x_0,\ldots, x_n$ define a smooth function on $M$. 
In \cite{BinosiPerotti2025} the following sliceness and slice-regularity criterion for polynomials was proved.

\begin{theorem}\cite[Theorems 21 and 28]{BinosiPerotti2025}\label{teo:poly_sliceness}
Assume that the Dunkl multiplicities satisfy $\sum_{i=1}^nk_i=(1-n)/2$. Let $f\in \Aa[x_0,\ldots,x_n]$ be a 
polynomial function with coefficients in $\Aa$. Then the following conditions are equivalent:
\begin{itemize} 
  \item[(i)] $f$ is a (left) slice polynomial, i.e., $f=\sum_{\alpha,\beta}x^\alpha (x^c)^\beta a_{\alpha,\beta}$;
  \item[(ii)] $f$ is in the kernel of the Casimir operator $\sC$; 
  \item[(iii)] $f$ is in the kernel of the spherical Dunkl-Dirac operator $\tGamma$. 
\end{itemize}
Moreover, the following conditions are equivalent:
\begin{itemize} 
  \item[(i)] $f$ is a slice-regular polynomial $f(x)=\sum_{j=0}^m x^ja_j$ with coefficients $a_j\in\Aa$;
  \item[(ii)] $f$ is Dunkl monogenic and belongs to the kernel of the Casimir operator $\sC$.
\end{itemize}
\end{theorem}
The Clifford algebra case for polynomials was proved in \cite{binosi2025dunklapproachsliceregular} using the Fischer decomposition associated to the Dunkl-Dirac operator.
In \cite{BinosiPerotti2025} we were able to extend the characterization of sliceness to any $C^1$ function on $\OO\subseteq M$. To achieve this, we need two more operators related to the Casimir operator.
Let $\OO\subseteq M$ be axially symmetric. We set
\[
\wsC:=\sum_{i=1}^nx_i T_{\B,i}-\Eu,\quad \sC':=\tfrac12\wsC\left(1+r\right),
\quad \sC'':=\sC'\x.
\]



\begin{theorem}\cite[Theorems 23 and 29]{BinosiPerotti2025}\label{teo:C1_sliceness}
Let $n\ge2$. Assume that the Dunkl multiplicities $k_i$ are non-positive, with at most one vanishing, and that $\sum_{i=1}^nk_i=(1-n)/2$. Let $\OO\subseteq M$ be open and axially symmetric and $f\in C^1(\OO,\Aa)$. Then the following conditions are equivalent:
\begin{itemize} 
  \item[(i)] $f$ is a slice function on $\OO$;
  \item[(ii)] $f$ is in the kernel of the operators $\sC$, $\sC'$ and $\sC''$. 
\end{itemize}
Moreover, the following conditions are equivalent:
\begin{itemize} 
  \item[(i)] $f$ is slice-regular on $\OO$;
  \item[(ii)] $f$ is Dunkl monogenic and belongs to the kernel of the operators $\sC$, $\sC'$ and $\sC''$. 
\end{itemize}
\end{theorem}

 The proof of Theorem \ref{teo:C1_sliceness} given in \cite{BinosiPerotti2025} is based on the Perron-Frobenius Theorem about spectral properties of non-negative matrices.

\section{Dunkl-regular function spaces} 
\label{sec:dunkl_regular_function_spaces}

\subsection{Intermediate Dunkl-Dirac operators} 
\label{sub:Intermediate_Dunkl-Dirac_operators}

Following the definitions made in the Clifford algebra case \cite{DeBieGenestVinet}, we introduce the intermediate $\zz_2^n$ Dunkl-Dirac operators of $M$ w.r.t.\ $\B$. 

\begin{definition}\label{def:DiracDunklA}
Let $A\subseteq[n]=\{1,\ldots,n\}$ and suppose that the open set $\OO\subseteq M$ is \emph{$A$-symmetric}, i.e.\ invariant w.r.t.\ reflections $r_i$ for all $i\in A$. Let $r_A=\prod_{i\in A}r_i$, $\x_A=\sum_{i\in A}x_iv_i$ and $\Eu_A=\sum_{i\in A}x_i\partial_{x_i}$.
We define 
\[\DD_A=\sum_{i\in A}v_i T_{\B,i},\quad \SA=\tfrac12\left(\left[\x_A,\DD_A\right]-1\right),
\]
\[
\wSA:=\sum_{i\in A}x_i T_{\B,i}-\Eu_A,\quad \SA':=\tfrac12\wSA\left(1+r_A\right),\quad \SA'':=\SA'\x_A
\]
and $\mathscr S_A=(\SA,\SA',\SA''):C^1(\OO,\Aa)\to (C^0(\OO,\Aa))^3$.
\end{definition}

These operators satisfy commutation relations similar to the case $A=[n]$, when $\x_A=\x$, $\DD_A=\DD_\B$ and $\SA=\sC$ \cite[Proposition 32]{BinosiPerotti2025}. 

In the following, for any $A\subseteq[n]$, we will always assume that the Dunkl multiplicities $k_i$ in $\DD_A$ and $\SA$ are non-positive, with $k_i=0$ for at most one index $i\in A$ and $\sum_{i\in A}k_i=(1-|A|)/2$. 
Under these conditions the operator $\SA$ can be written, as in Proposition \ref{pro:operators}, in the form
\begin{equation}\label{eq:SA}
  \SA=\x_A\DD_A+\Eu_A,
\end{equation}
and it holds $\mathscr S_{\{i\}}=0$ for every $i=1,\ldots,n$.  

\begin{definition}
Let $\emptyset\not=A\subseteq[n]$ and let $\OO\subseteq M$ be $A$-symmetric. 
Given a function $f:\OO\to\Aa$, we define the \emph{$A$-spherical value} of $f$ as the function $f^\circ_{s,A}:\OO\to\Aa$ defined by
\[
f^\circ_{s,A}(x)=\tfrac12(f(x)+f(\bar x^A)).
\]
Here $\bar x^A=r_A(x)\in\OO$ is the reflected point w.r.t.\ $A$.
\end{definition}

\begin{remark}\label{rem:sliceness_continuity_A}
By definitions, the equality $f=f^\circ_{s,A}+\x_A^{-1}{(\x_A f)}^\circ_{s,A}$ holds on $\OO\setminus\{\x_A=0\}$.
If $f\in C(\OO,\Aa)$, then the function $g:=\|\x_A\|^{-1}(\x_Af)^\circ_{s,A}$ extends continuously to $\OO$ with value zero on $\OO\cap\{\x_A=0\}$ (see the proof of \cite[Proposition 9]{BinosiPerotti2025}).
\end{remark}

Let $A=\{i_1,\ldots,i_\ell\}\subseteq[n]$. Let $M_A$ be the hypercomplex subspace $M_A:=\langle 1,v_{i_1},\ldots,v_{i_\ell}\rangle\subseteq M$ and $\s_{M_A}=\s_\Aa\cap M_A$ the $(\ell-1)$-dimensional unit sphere in $M_A\cap\ker(t)$.

\begin{definition}
Let $\emptyset\ne A\subseteq[n]$ and $A^c=[n]\setminus A=\{j_1,\ldots,j_{n-\ell}\}$ with $j_1<\cdots <j_{n-\ell}$. A subset $\OO$ of $M$ is called \emph{$A$-circular} if for any $x=x_0+\x_A+\x_{A^c}\in\OO$, with $0\ne\x_A=\beta I$ and $I\in \s_{M_A}$, $\beta\in\rr$,  the element $x_J:=x_0+\beta J+\x_{A^c}$ of $M$ belongs to $\OO$ for every $J\in \s_{M_A}$. 
Equivalently, $\OO$ is $A$-circular if there exists $E=E'\times E''\subseteq\cc\times\rr^{n-\ell}$ such that $\OO=\OO_E$, where 
\[\textstyle
\OO_E:=\left\{x_0+\beta J+\sum_{i=1}^{n-\ell}v_{j_i}y_i\;|\; (x_0+i\beta,y)\in E, J\in \s_{M_A}\right\}.
\]
\end{definition}

Every $A$-circular set is $A$-symmetric. 
A set $\OO$ is axially symmetric if and only if it is $[n]$-circular. Therefore $\OO$ is $A$-circular in $M$ if and only if $\OO\cap M_A$ is axially symmetric in $M_A$. 





\subsection{Dunkl-regular functions spaces}
\label{sub:Dunkl-regular_function_spaces}

Given a partition  $\PP=\{A_1,\ldots,A_\ell\}$ of the set $[n]$, a \emph{$\PP$-admissible} sequence of Dunkl multiplicities is an $n$-tuple ${\bf k}=(k_1,\ldots,k_n)$ of non-positive real numbers such that, for every $j=1,\ldots,\ell$, it holds $2\sum_{i\in A_j}k_i=1-|A_j|$ and $k_i=0$ for at most one index $i\in A_j$. 

Let $\B=(1,v_1,\ldots,v_n)$ be a hypercomplex basis of the hypercomplex subspace $M$ of $\Aa$, with  coordinates $x_0,x_1,\ldots,x_n$. 

\begin{definition}\label{def:P-Dunkl-regular}
Let $\OO\subseteq M$ be open and $\zz_2^n$-invariant. We call \emph{$\PP$-Dunkl-regular 
functions} on $\OO$ the elements of the real vector space (a right $\Aa$-module if $\Aa$ is associative)
\[
\F_{\PP}(\OO)=\F_{A_1,\ldots,A_\ell}(\OO):=\{f\in C^1(\OO,\Aa)\;|\; f\in\ker D_\PP\cap\ker \mathscr S_\PP\},
\] 
where $D_\PP$ is the \emph{Dunkl-Cauchy-Riemann operator} 
\begin{equation}\label{eq:DP}
  D_\PP=\partial_{x_0}+\sum_{j=1}^\ell\DD_{A_j}=\dM+\sum_{i=1}^n \frac{k_iv_i}{x_i}(1-r_i),
\end{equation}
and $\mathscr S_\PP=(\mathscr S_{A_1},\ldots,\mathscr S_{A_\ell})$. We say that a function is \emph{Dunkl-regular} if it is $\PP$-Dunkl-regular for some $\PP$. 
The \emph{conjugated Dunkl-Cauchy-Riemann operator} $D^c_\PP$ is defined as
$D^c_\PP=\partial_{x_0}-\DD_\PP$, where $\DD_\PP=\sum_{j=1}^\ell\DD_{A_j}$ is the \emph{Dunkl-Dirac operator}. 
\end{definition}

It holds $D_\PP D^c_\PP=D^c_\PP D_\PP=\Delta_{D,M}$, the Dunkl Laplacian. Since $\F_{\PP}(\OO)\subseteq\ker D_\PP$, every Dunkl-regular function is Dunkl monogenic and  Dunkl harmonic  on $\OO$.


\begin{examples}\label{ex:FP}We give some examples of Dunkl-regular function spaces.
\begin{itemize}
  \item[(i)]
  $\F_{\{1\},\ldots,\{n\}}(\OO)=\ker\dM=\M(\OO)$  on any open subset $\OO$ of $M$.
  \item[(ii)]
  $\F_{[n]}(\OO)=\sr(\OO)$ on any axially symmetric open subset $\OO$ of $M$.
  \item[(iii)]
  On the Clifford algebra $\rr_n$, the function space $\F_\PP(\OO)$, with partition \[\PP=\{\{1\},\ldots,\{p\},\{p+1,\ldots,n\}\}\] is the space of \emph{generalized partial-slice monogenic functions of type $(p,n-p)$} \cite{Sabadini_Xu_TAMS}.  
  \item[(iv)]
Given a sequence $T=(t_0,\ldots,t_\tau)$ of integers with $\tau>0$ and $0\le t_0<t_1<\cdots <t_\tau=n$, let $\PP$ be the partition
\[
\PP=\{\{1\},\ldots,\{t_0\},\{t_0+1,\ldots,t_1\},\{t_1+1,\ldots,t_2\},\ldots,\{t_{\tau-1}+1,\ldots,t_\tau\}\}.
\]
If $\OO$ is a $\PP$-circular open subset of $M$, i.e., $\OO$ is $A_j$-circular for every $j=1,\ldots,\ell$,  then $\F_\PP(\OO)$  is the set of \emph{$T$-regular functions} of class $C^1$ on $\OO$ \cite{GhiloniStoppato_arXiv24}. 
\end{itemize}
\end{examples}


Let $\PP=\{A_1,\ldots,A_\ell\}$ be a partition of the set $[n]$. Then $\PP$ defines a partition of the number $n$ as $n=|A_1|+\cdots+|A_\ell|$. 
The \emph{number of partitions} of the set $[n]=\{1,\ldots,n\}$ is the \emph{Bell number} $B_n$, while the number of partitions of the integer $n$ is the \emph{partition number} $p(n)$.

\begin{theorem}\cite[Theorem 60]{BinosiPerotti2025}\label{teo:classification1}
Let $n=\dim M-1$ and let $\OO\subseteq M$ be open and $\zz_2^n$-invariant. 
The function space $\F_\PP(\OO)$ does not depend on the Dunkl multiplicities of the operators $D_\PP$ and $\mathscr S_\PP$, provided that they are $\PP$-admissible. Given two different partitions $\PP$ and $\PP'$ of the set $[n]$, the function spaces $\F_\PP(\OO)$ and $\F_{\PP'}(\OO)$ are distinct. 
\end{theorem}

\begin{corollary}\label{cor:Kreg}
Let $\PP=\{A_1,\ldots,A_\ell\}$. The Dunkl multiplicities $k_i$ of the operators $D_\PP$, $\mathscr S_\PP$ defining the space $\F_\PP(\OO)$ can be taken as $k_i=-\tfrac12+\frac1{2|A_j|}\in(-\tfrac12,0]$ for every $i\in A_j$ and $j=1,\ldots,\ell$. Another possible choice of multiplicities is obtained taking $k_i\in\{0,-1/2\}$, with exactly one $\alpha_j\in A_j$, for every $j=1,\ldots,\ell$, such that $k_{\alpha_j}=0$. 
\end{corollary}

In view of the Corollary, the Dunkl multiplicities of the operators $D_\PP$, $\mathscr S_\PP$ in $\F_\PP(\OO)$, although non-positive, can be chosen in the \emph{regular set} $K^{reg}$ (see e.g.\ \cite[\S2.4]{Rosler}) for the reflection group $\zz_2^n$.  

Let $\sigma$ be a permutation of the set $[n]$. Given a partition $\PP=\{A_1,\ldots,A_\ell\}$ of $[n]$, let $\PP_\sigma=\{B_1,\ldots,B_\ell\}$ be the partition defined by $B_i=\sigma^{-1}(A_i)$, $i=1,\ldots,\ell$. 
A sequence ${\bf k}=(k_1,\ldots,k_n)$ of Dunkl multiplicities is $\PP$-admissible if and only if the permuted sequence ${\bf k}_\sigma=(k_{\sigma(1)},\ldots,k_{\sigma(n)})$ is $\PP_\sigma$-admissible. 
Let $\B_\sigma=(1,v'_1,\ldots,v'_{n})$ be the basis obtained by $\B$ permuting its elements: $v'_i=v_{\sigma(i)}$ for every $i=1\ldots,n$. 

In general, the function space $\F_\PP(\OO)$ depends on the choice of $\B$. If we want to highlight this dependence, we write $\F_\PP^\B(\OO)$ instead of $\F_\PP(\OO)$.
We call \emph{equivalent} two function spaces $\F^\B_\PP(\OO)$ and $\F^{\B'}_{\PP'}(\OO)$  if $\B'=\B_\sigma$ and $\PP'=\PP_\sigma$ for some permutation $\sigma$.

\begin{theorem}\cite[Theorem 63]{BinosiPerotti2025}\label{teo:classification2}
Let $n=\dim M-1$ and let $\OO\subseteq M$ be open and $\zz_2^n$-invariant. Let $\B$ be a fixed hypercomplex basis.
Two spaces $\F_\PP^\B(\OO)$ and $\F_{\PP'}^{\B'}(\OO)$ of Dunkl-regular functions are equivalent (and isomorphic as real vector spaces) if and only if $\PP$ and $\PP'$ define the same partition of $n=\dim M-1$. 
\end{theorem}

The number of non-equivalent functions spaces $\F_\PP(\OO)$ is equal to the partition number $p(n)$ (see Table \ref{table}). 
The spaces $\M(\OO)$ and $\sr(\OO)$ are the unique functions spaces $\F_\PP^\B(\OO)$ that are not equivalent to any other space $\F_{\PP'}^{\B'}(\OO)$. 

\begin{table}[h]
\begin{center}
\begin{tabular}{c|c|c|c|c|c|c|c|c|c}
\hline &&&&&&&&&\\[-10pt]
$n=\dim(M)-1$ &$1$&$2$&$3$&$3$&$3$&$4$&$5$&7&\dots \\
\hline &&&&&&&&&\\[-10pt]
$p(n)$ & $1$ & 2&3&3&3&5&7&15&\dots\\
\hline &&&&&&&&&\\[-10pt]
$q(n)$ & $1$ & 1&2&2&2&3&4&5&\dots\\
\hline &&&&&&&&&\\[-10pt]
$B_n$ & $1$ & $2$ &$5$&$5$&$5$&$15$&$52$&$877$&\dots\\
\hline &&&&&&&&&\\[-10pt]
$M$&$\cc$ &$\hh_r$ & $\hh$ &$\hh$&$\rr^4$&$\rr^5$ & $\rr^6$& $\oo$&\dots\\
\hline &&&&&&&&&\\[-10pt]
$\Aa$&$\cc$ &$\hh$ & $\hh$ &$\oo$&$\rr_3$&$\rr_4$ & $\rr_5$& $\oo$&\dots\\
\hline
\end{tabular}
\vskip5pt
\caption{Partition numbers and Bell numbers for some hypercomplex subspaces $M$. The function $q(n)$ counts the odd partitions of $n$ (see Remark \ref{rem:qn}).}
\label{table}
\end{center}
\end{table}

\subsection{$\PP$-slice functions}\label{sec:Pslice}

\begin{definition}\label{def:pslice}
Let $\PP=\{A_1,\ldots,A_\ell\}$ be a partition of the set $[n]$. 
Let $\OO=\OO_D$ be a $\PP$-circular open subset of $M$ (i.e., $\OO$ is $A_j$-circular for every $j=1,\ldots,\ell$).  A function $f\in C^1(\OO)$ is called \emph{$\PP$-slice} if it belongs to $\ker\mathscr S_\PP=\cap_{j=1}^\ell\ker\mathscr S_{A_j}$. 
\end{definition}

Observe that every $C^1$ function $f:\OO\to\Aa$ is $\PP$-slice when $\PP=\{\{1\},\ldots,\{n\}\}$.

The next proposition characterizes functions in the kernel of the differential-difference operator $\mathscr S_A$.

\begin{proposition}\cite[Proposition 46]{BinosiPerotti2025}\label{pro:SA}
Let $\emptyset\ne A\subseteq[n]$ and $A^c=[n]\setminus A=\{j_1,\ldots,j_{n-\ell}\}$ with $j_1<\cdots <j_{n-\ell}$. Let $\OO=\OO_E$ be an $A$-circular open subset of $M$, with $E=E'\times E''\subseteq\cc\times\rr^{n-\ell}$.
If $f\in C^1(\OO,\Aa)$ and $\mathscr S_A f=0$ on $\OO$, then there exist continuous functions 
$F_{\emptyset}^A,F_1^A:E\to\Aa$ such that for every $x\in\OO$ it holds
\[
f(x)=F_{\emptyset}^A(z,y)+JF_1^A(z,y).
\]
Here $z=x_0+i\beta\in E'$, $y\in E''$, $x=x_0+\beta J+\sum_{i=1}^{n-\ell}v_{j_i}y_i\in\OO$, where $J=\text{sign}(\beta)\tfrac{\x_A}{\|\x_A\|}=\beta^{-1}\x_A$ if $|\beta|=\|\x_A\|\ne0$, and $J$ is any element of $\s_{M_A}$ if $\beta=0=\x_A$. The functions $F_{\emptyset}^A,F_1^A$ are, respectively, even and odd w.r.t.\ $\beta$.
Conversely, if $f(x)=F_{\emptyset}(z,y)+JF_1(z,y)$ with $F_{\emptyset},F_1$ an even/odd pair w.r.t.\ $\beta$, then 
$\mathscr S_A f=0$ on $\OO$.
\end{proposition}

\begin{remark}\label{rem:even-odd}
If $f(x)=F_{\emptyset}^A(z,y)+JF_1^A(z,y)$ with $F_{\emptyset}^A,F_1^A$ as in the Proposition, 
then $f^\circ_{s,A}(x)=F_{\emptyset}^A(z,y)$ and $(\x_Af)^\circ_{s,A}(x)=-\beta F_1^A(z,y)$. 
In particular, $f^\circ_{s,A}$ and ${(\x_A f)}^\circ_{s,A}$ are constant on $\s_x\cap M_A$ for every $x\in\OO\cap M_A$. Moreover, this shows that $F_{\emptyset}^A$ and $F_1^A$ are uniquely determined by $f$.
\end{remark}

Let $\B=\{1,v_1,\ldots,v_n\}$ be a fixed hypercomplex basis of $M$, with corresponding real coordinated $x_0,x_1,\ldots,x_n$.  
For any $i=1,\ldots,n$, define the operators
\begin{equation}\label{eq:delta12}
\delta_1^{i}:=\frac{1-r_i}{x_i},\quad \delta_2^{i}:=\frac{1-r_i}{x_i^2}-\frac2{x_i}\partial_{x_i}.
\end{equation}

\begin{proposition}\label{pro:delta12}
Let $\OO\subseteq M$ be open and $\zz_2^n$-invariant. For every $i=1,\ldots,n$, it holds $T_{\B,i}=\partial_{x_i}+k_i\delta_1^i$ and $x_i\delta_2^i=\delta_1^i-2\partial_{x_i}$. Moreover, it holds
\[
\delta_2^i\delta_1^i+[\partial_{x_i},\delta_2^i]=0.
\]
\end{proposition}
\begin{proof}
The first statement is immediate from definitions. We prove the second one. Let $f\in C^2(\OO,\Aa)$. Then, for $x_i\ne0$, it holds
\[
\delta_2^i\delta_1^if=\delta_2\left(\frac{f-r_if}{x_i}\right)=-\frac2{x_i}\partial_{x_i}\left(\frac{f-r_if}{x_i}\right)=2\frac{f-r_if}{x_i^3}-\frac2{x_i^2}\left(\partial_{x_i}f+r_i\partial_{x_i}f\right),
\]
since $\{\partial_{x_i},r_i\}=0$. On the other hand,
\begin{align*}
[\partial_{x_i},\delta_2^i]f&=\partial_{x_i}\left(\frac{f-r_if}{x_i^2}\right)+\frac2{x_i^2}\partial_{x_i}f-\frac2{x_i}\partial_{x_i}^2f-\frac{\partial_{x_i}f-r_i\partial_{x_i}f}{x_i^2}+\frac2{x_i}\partial_{x_i}^2f\\
&=-2\frac{f-r_if}{x_i^3}+\frac2{x_i^2}\partial_{x_i}f+\frac2{x_i^2}r_i\partial_{x_i}f=-\delta_2^i\delta_1^if.
\end{align*}
\end{proof}

\begin{proposition}\label{pro:delta1_2slice}
Let $A\subseteq[n]$, $\OO\subseteq M$ be open and $A$-circular. Assume that $\mathscr S_A f=0$ on $\OO$ and let $i\in A$. 
Then the functions $v_i\delta_1^if$ are of class $C^1$ on $\OO'=\{x\in\OO\,|\, \x_A\ne0\}$, are independent of $i\in A$ and belong to $\ker\mathscr S_A$. 
Moreover, the $A$-spherical value $f^\circ_{s,A}$ and the function $\x_A f$ belong to $\ker\mathscr S_A$. 
If $f\in C^2(\OO,\Aa)$, then the functions $\delta_2^if$ are of class $C^1$ on $\OO'$, are independent of $i\in A$ and belong to $\ker\mathscr S_A$. Moreover, also the functions $\dM f$ and $\partial_M f$ belong to $\ker\mathscr S_A$.
\end{proposition}
\begin{proof}
Let $\OO=\OO_E$, with $E=E'\times E''\subseteq\cc\times\rr^{n-\ell}$, $\ell=|A|$, $A^c=[n]\setminus A=\{j_1,\ldots,j_{n-\ell}\}$ with $j_1<\cdots <j_{n-\ell}$. Since $\mathscr S_Af=0$, Proposition \ref{pro:SA} provides continuous functions 
$F_{\emptyset}^A,F_1^A:E\to\Aa$ such that for every $x\in\OO'$ it holds
\[
f(x)=F_{\emptyset}^A(z,y)+JF_1^A(z,y),
\]
where $z=x_0+i\beta\in E'\setminus\rr$, $y=(y_1,\ldots,y_{n-\ell})\in E''$, $x=x_0+\beta J+\sum_{i=1}^{n-\ell}v_{j_i}y_i\in\OO$, with $J=\beta^{-1}\x_A$ and the functions $F_{\emptyset}^A,F_1^A$ are, respectively, even and odd w.r.t.\ $\beta$. If $i\in A$, it holds
\[
v_i\delta_1^i f(x)=\frac{v_i}{x_i}\left(F_{\emptyset}^A(z,y)+JF_1^A(z,y)-\big(F_{\emptyset}^A(z,y)+J'F_1^A(z,y)\big)\right)
=\frac{v_i}{x_i}(J-J')F_1^A(z,y)).
\]
where $J'=r_i(J)$. Therefore $J-J'=2x_iv_i/\beta$ and $v_i\delta_1^i f(x)=-2F_1^A(z,y)/\beta$, independent of $i\in A$.  
Since $F_1^A(z,y)=-\beta^{-1}(\x_Af)^\circ_{s,A}(x)$, the functions $v_i\delta_1^i f$ are of class $C^1$ on $\OO'$. Since $F_1^A(z,y)/\beta$ is even w.r.t.\ $\beta$, the converse part of Proposition \ref{pro:SA} shows that $v_i\delta_1^i f$  belongs to $\ker\mathscr S_A$. 
The equalities
\[
f^\circ_{s,A}=\tfrac12(f(x)+f(\bar x^A))=\tfrac12\big((F_{\emptyset}^A(z,y)+JF_1^A(z,y))+(F_{\emptyset}^A(z,y)-JF_1^A(z,y)\big)=F_{\emptyset}^A(z,y)
\]
and
\begin{align*}
\x_A f(x)&=(x_0+\beta J)(F_{\emptyset}^A(z,y)+JF_1^A(z,y))\\
&=((x_0F_{\emptyset}^A(z,y)-\beta F_1^A(z,y))+J(\beta F_{\emptyset}^A(z,y)+x_0F_1^A(z,y))
\end{align*}
and Proposition \ref{pro:SA} show that $f^\circ_{s,A}$ and $\x_A f(x)$ are in $\ker\mathscr S_A$. 
From the equalities  $J=\beta^{-1}\x_A$ and $\beta^2=\|\x_A\|^2$, we get that
\begin{equation}\label{eq:dfi}
\partial_{x_i}f=\frac{x_i}\beta\partial_\beta F^A_\emptyset(z,y)+\left(\frac{v_i}\beta-J\frac{x_i}{\beta^2}\right)F^A_1(z,y)+J\frac{x_i}\beta\partial_\beta F^A_1(z,y)
\end{equation}
for all $i\in A$. This yields 
\begin{align*}
\delta_2^i f&=\frac1{x_i^2}\frac{2x_iv_i}{\beta}F^A_1(z,y)-\frac2{x_i}\left(\frac{x_i}\beta\partial_\beta F^A_\emptyset(z,y)+\left(\frac{v_i}\beta-J\frac{x_i}{\beta^2}\right)F^A_1(z,y)+J\frac{x_i}\beta\partial_\beta F^A_1(z,y)\right)\\
&=
-2\beta^{-1}\partial_\beta F^A_\emptyset(z,y)+J\left(2\beta^{-2}F^A_1(z,y) -2\beta^{-1}\partial_\beta F^A_1(z,y)\right)=:G_{\emptyset}^A(z,y)+JG_1^A(z,y),
\end{align*}
with $G^A_\emptyset$ and $G^A_1$ respectively even and odd w.r.t.\ $\beta$ and independent of $i$. Proposition \ref{pro:SA} shows that $\delta_2^i f$  belongs to $\ker\mathscr S_A$. We now prove the last statement. Since
\[
\dM f=\sum_{i\in A}v_i\partial_{x_i}f+\big(\partial_{x_0}f+\sum_{j\in A^c}v_j\partial_{x_j}f\big),
\]
it is sufficient to prove that $f_1:=\sum_{i\in A}v_i\partial_{x_i}f$ and $f_2:=\partial_{x_0}f+\sum_{j\in A^c}v_j\partial_{x_j}f$ are in $\ker\mathscr S_A$. From \eqref{eq:dfi}, \eqref{eq:anticommute} and the equality $\x_A=\beta J$ we get
\begin{align*}
f_1(x)&=\sum_{i\in A}\left(v_ix_i\beta^{-1}\partial_\beta F^A_\emptyset(z,y)-\beta^{-1}F^A_1(z,y)-v_ix_i\left(
J\left(\beta^{-2}F^A_1(z,y)-\beta^{-1}\partial_\beta F^A_1(z,y)\right)\right)\right)\\
&=\x_A\beta^{-1}\partial_\beta F^A_\emptyset(z,y)-|A|\beta^{-1}F^A_1(z,y)-\x_A\left(
J\left(\beta^{-2}F^A_1(z,y)-\beta^{-1}\partial_\beta F^A_1(z,y)\right)\right)\\
&=J\partial_\beta F^A_\emptyset(z,y)-|A|\beta^{-1}F^A_1(z,y)+\left(\beta^{-1}F^A_1(z,y)-\partial_\beta F^A_1(z,y)\right)\\
&=\left((1-|A|)\beta^{-1}F^A_1(z,y)-\partial_\beta F^A_1(z,y)\right)+J\,\partial_\beta F^A_\emptyset(z,y)=:H^A_\emptyset(z,y)+JH^A_1(z,y),
\end{align*}
with $H^A_\emptyset$, $H^A_1$ an even-odd pair w.r.t.\ $\beta$, proving that $f_1\in\ker\mathscr S_A$. To conclude, we observe that also $f_2\in\ker\mathscr S_A$, since it can be written, for $x\in\OO'$, as
\begin{align*}
f_2(x)&=\partial_{x_0}F_{\emptyset}^A(z,y)+J\partial_{x_0}F_1^A(z,y)+\sum_{j\in A^c}v_j(\partial_{x_j}F_{\emptyset}^A(z,y)+J\partial_{x_j}F_1^A(z,y))\\
&=\big(\partial_{x_0}F_{\emptyset}^A(z,y)+\sum_{j\in A^c}v_j \partial_{x_j}F_{\emptyset}^A(z,y)\big)+J\big(\partial_{x_0}F_1^A(z,y)-\sum_{j\in A^c}v_j\partial_{x_j}F_1^A(z,y)\big),
\end{align*}
where we used the equality $v_j(Ja)=-J(v_ja)$, valid for every $j\in A^c$, $a\in\Aa$, a consequence of \eqref{eq:anticommute}. The same argument proves that $\partial_M f=2\partial_{x_0}f-\dM f\in\ker\mathscr S_A$. 
\end{proof}

The function, independent of $i\in A$,  
\[
f'_{s,A}:=-\frac12v_i\delta_1^if=-v_{i}\frac{f-r_{i}f}{2x_{i}}=(x-\bar x^i)^{-1}(f-r_if),
\]
where $\bar x^i=r_i(x)$, is called \emph{$A$-spherical derivative} of $f\in\ker\mathscr S_A$. The name is justified by the following property: if $f$ is a slice function, then $A=[n]$ and $f'_{s,A}$ is the usual spherical derivative of a slice function. In the notation of the proof above, it holds $f'_{s,A}(x)=-2^{-1}v_i\delta_1^if(x)=\beta^{-1}F^A_1(z,y)=-\beta^{-2}(\x_A f)^\circ_{s,A}(x)=-\|\x_A\|^{-2}(\x_A f)^\circ_{s,A}(x)$. Since $\x_A^{-1}=-\|\x_A\|^{-2}\x_A$, for every $f\in\ker\mathscr S_A$ it holds, where $\x_A\ne0$ (see Remark \ref{rem:sliceness_continuity_A}), 
\begin{equation}\label{eq:fsA}
f=f^\circ_{s,A}+\x_A f'_{s,A}.
\end{equation}
The previous Proposition says that if $f\in\ker\mathscr S_A$, then also $f'_{s,A}$ belongs to $\ker\mathscr S_A$. 

\begin{corollary}\label{cor:delta1_2Pslice}
Let $\PP=\{A_1,\ldots,A_\ell\}$ and let ${\bf k}=(k_1,\ldots,k_n)$ be a sequence of $\PP$-admissible Dunkl multiplicities.  Let $\OO\subseteq M$ be open and $\PP$-circular. 
If $f\in C^2(\OO,\Aa)$ and $\mathscr S_\PP f=0$, then for any $j=1,\ldots,\ell$, the $A_j$-spherical derivative $f'_{s,A_j}=-\frac12v_{\alpha_j}\delta_1^{\alpha_j}f$ and the function $\delta_2^{\alpha_j}f$
are $\PP$-slice functions of class $C^1$ on $\OO'=\bigcap_{j=1}^\ell\{x\in\OO\,|\, \x_{A_j}\ne0\}$, independent of the choice of $\alpha_j\in A_j$. 
\end{corollary}
\begin{proof}
From Proposition \ref{pro:delta1_2slice}, we have $\mathscr S_{A_j}(f'_{s,A_j})=0=\mathscr S_{A_j}(\delta_2^{\alpha_j}f)$ on $\OO'$. 
For every $i\ne j$, it is easy to verify that the commutation relations 
\[
\big[\underline S_{A_i},\delta_1^{\alpha_j}\big]=\big[\x_{A_i}\DD_{A_i}+\Eu_{A_i},\delta_1^{\alpha_j}\big]=0,\quad \big[\widetilde{\underline S}'_{A_i},\delta_1^{\alpha_j}\big]=\big[\widetilde{\underline S}''_{A_i},\delta_1^{\alpha_j}\big]=0
\]
and
\[
\big[\underline S_{A_i},\delta_2^{\alpha_j}\big]=\big[\x_{A_i}\DD_{A_i}+\Eu_{A_i},\delta_2^{\alpha_j}\big]=0,\quad \big[\widetilde{\underline S}'_{A_i},\delta_2^{\alpha_j}\big]=\big[\widetilde{\underline S}''_{A_i},\delta_2^{\alpha_j}\big]=0
\]
hold, and then $[\mathscr S_{A_i},\delta_1^{\alpha_j}]=[\mathscr S_{A_i},\delta_2^{\alpha_j}]=0$. This implies that $\mathscr S_\PP(\delta_2^{\alpha_j}f)=0$. 
Since $\{\DD_{A_i},v_{\alpha_j}\}=0=\{\x_{A_i},v_{\alpha_j}\}$ when $i\ne j$, we deduce also that $
\big[\mathscr S_{A_i},v_{\alpha_j}\delta_1^{\alpha_j}\big]=0$  for $i\ne j$. Therefore $\mathscr S_\PP(f'_{s,A_j})=0$ on $\OO'$ for every $j=1,\ldots,\ell$. 
\end{proof}

\begin{proposition}\label{pro:DM}
Let $\PP=\{A_1,\ldots,A_\ell\}$ and let ${\bf k}=(k_1,\ldots,k_n)$ be a sequence of $\PP$-admissible Dunkl multiplicities. Let $\OO\subseteq M$ be open and $\PP$-circular and $f\in C^1(\OO,\Aa)$. Then it holds:
\begin{itemize}
  \item[(i)]
  If $f$ is $\PP$-slice, i.e., $\mathscr S_\PP(f)=0$ on $\OO$, then 
\begin{equation}\label{eq:dp}
D_\PP f=\dM f+\sum_{j=1}^\ell(|A_j|-1)f'_{s,A_j}.
\end{equation}
\item[(ii)]
  If $f\in\F_\PP(\OO)$, then 
\begin{equation}\label{eq:dm}
\dM f=
\sum_{j=1}^\ell(1-|A_j|)f'_{s,A_j}.
\end{equation}
\end{itemize}
\end{proposition}
\begin{proof}
If $\mathscr S_\PP f=0$, Corollary \ref{cor:delta1_2Pslice} states that the functions $f'_{s,A_j}=-\frac12v_{\alpha_j}\delta_1^{\alpha_j}f$ are $\PP$-slice functions, independent of the choice of $\alpha_j\in A_j$. 
Using Proposition \ref{pro:delta12}, we can write 
\begin{align*}
D_\PP f&=\dM f+\sum_{j=1}^\ell\sum_{i\in A_j}k_iv_i\delta_1^i f=\dM f-2\sum_{j=1}^\ell\sum_{i\in A_j}k_i f'_{s,A_j}\\
&=\dM f+\sum_{j=1}^\ell(|A_j|-1)f'_{s,A_j}
\end{align*}
and \eqref{eq:dp} is proved. Formula \eqref{eq:dm} follows immediately from \eqref{eq:dp} when $f\in\F_\PP(\OO)$. 
\end{proof}

If $f$ is slice-regular on $\OO$, i.e., $f\in\F_{[n]}(\OO)$, then $\ell=1$ and the function $f'_{s,A_1}$ is the usual spherical derivative of a slice function. In this case equations \eqref{eq:dp}, \eqref{eq:dm} were proved in \cite[Proposition\ 9]{CRoperators}.

\begin{corollary}\label{cor:Pslice}
Let ${\bf k}=(k_1,\ldots,k_n)$ be a sequence of $\PP$-admissible Dunkl multiplicities. 
Let $\OO\subseteq M$ be open and $\PP$-circular.  
If $f\in C^2(\OO,\Aa)$ is $\PP$-slice, then the functions $D_\PP f$, $D^c_\PP f$, $\DD_\PP f$, $\dM f$, $\partial_M f$ are $\PP$-slice on $\OO$. If $f\in C^3(\OO,\Aa)$ is $\PP$-slice, then the Laplacian $\Delta_M f$ and the Dunkl-Laplacian $\Delta_{D,M}f$ are $\PP$-slice on $\OO$. 
\end{corollary}
\begin{proof}
Let $\PP=\{A_1,\ldots,A_\ell\}$. 
From Proposition \ref{pro:delta1_2slice}, it follows that $\dM f$, $\partial_M f$ are $\PP$-slice on $\OO$. From Corollary \ref{cor:delta1_2Pslice}, formula \eqref{eq:dp} and the commutation property $[\partial_{x_0},\mathscr S_{A_j}]=0$, we get that also $D_\PP f$, $\DD_\PP f$ and $D^c_\PP f=\partial_{x_0}f-\DD_\PP f$ are $\PP$-slice. Then also $\Delta_{M}f=\partial_M (\dM f)$  and $\Delta_{D,M}f=D^c_\PP(D_\PP f)$ are $\PP$-slice. 
\end{proof}

We now show that the Dunkl Laplacian $\Delta_{D,M}$ of a $\PP$-slice function can be expressed through the operators $\delta_2^i$.

\begin{proposition}\label{pro:DeltaM}
Let $\OO\subseteq M$ be open and $\zz_2^n$-invariant. Let $\PP=\{A_1,\ldots,A_\ell\}$. Then it holds:
\begin{itemize}
  \item[(i)]
  If $f$ is $\PP$-slice, i.e., $\mathscr S_\PP(f)=0$ on $\OO$, then, for any choice of $\alpha_j\in A_j$,
\begin{equation}\label{eq:DeltaDM_Pslice}
\Delta_{D,M} f=\Delta_M f+\sum_{j=1}^\ell \tfrac{|A_j|-1}2\delta_2^{\alpha_j} f.
\end{equation}
\item[(ii)]
  If $f\in\F_\PP(\OO)$, then, for any choice of $\alpha_j\in A_j$,
\begin{equation}\label{eq:DeltaM}
\Delta_M f=\sum_{j=1}^\ell \tfrac{1-|A_j|}2\delta_2^{\alpha_j} f.
\end{equation}
\end{itemize}
\end{proposition}
\begin{proof}
The Dunkl Laplacian $\Delta_{D,M}$ on $M$ associated with the reflection group $\zz_2^n$ has the form 
given in equation \eqref{eq:DeltaDM}: 
\[
\Delta_{D,M}=\Delta_M+\sum_{i=1}^nk_i\left(\frac2{x_i}\partial_{x_i}-\frac{1-r_i}{x_i^2}\right)=\Delta_M-\sum_{i=1}^nk_i\delta_2^i,
\]
with $(k_1,\ldots,k_n)$ any $\PP$-admissible sequence of Dunkl multiplicities. 
If $f$ is $\PP$-slice, then $\delta_2^i f$ is independent of the choice of $i\in A_j$ (Proposition \ref{pro:delta1_2slice}). Therefore, for any $\alpha_j\in A_j$,
\[
\Delta_{D,M}f=\Delta_M f-\sum_{i=1}^nk_i\delta_2^i f=\Delta_M f-\sum_{j=1}^\ell \sum_{i\in A_j}k_i\delta_2^{\alpha_j} f=\Delta_M f+\sum_{j=1}^\ell \tfrac{|A_j|-1}2\delta_2^{\alpha_j} f.
\]
Point (ii) is an immediate consequence of (i), since every $\PP$-Dunkl-regular function is Dunkl harmonic. 
\end{proof}

\begin{remark}\label{rem:DM_delta1}
Any choice of elements $\alpha_j\in A_j$, for $j=1,\ldots,\ell$, provides a $\PP$-admissible sequence $(k_1,\ldots,k_n)$ with the maximum number $\ell=|\PP|$ of zero multiplicities, 
by setting
\begin{equation}\label{eq:canonical}
\begin{cases}
k_i=0&\text{\quad if $i\in\{\alpha_1,\ldots,\alpha_\ell\}$},\\
k_i=-1/2&\text{\quad otherwise.}
\end{cases}
\end{equation}
With this choice of multiplicities, the Dunkl-Cauchy-Riemann operator $D_\PP$ can be written as
\[
  D_\PP=\dM-\tfrac12\sum_{\substack{i=1\\i\not\in\alpha}}^n v_i\frac{1-r_i}{x_i}=\dM-\tfrac12\sum_{\substack{i=1\\i\not\in\alpha}}^n v_i\delta_1^i.
\]
\end{remark}

\subsection{Cauchy-Kovalevskaya extension. Dunkl-regular polynomials 
}\label{sec:CK}

Cauchy-Kovalevskaya extension is a well-known method to construct functions in the kernel of a given differential operator. It was applied in \cite[Theorem 4.6]{Huo_Ren_Xu_arXiv25} to construct monogenic polynomials on a hypercomplex subspace of an alternative algebra. Here we generalized that construction to Dunkl-regular polynomials.

\begin{theorem}\label{teo:CK}
Let $\PP=\{A_1,\ldots,A_\ell\}$. 
Let $g\in\Aa[x_1,\ldots,x_n]$ be a $\PP$-slice polynomial function, i.e., a polynomial in the kernel of 
$\mathscr S_{A_1}\ldots,\mathscr S_{A_\ell}$. Then the Cauchy-Kovalevskaya extension (or CK-extension) of $g$ induced by the operator $D_\PP=\partial_{x_0}+\DD_\PP$, defined as
\[
CK[g](x):=\sum_{k=0}^{\deg (g)}\frac{(-x_0)^k}{k!}\DD_\PP^k g(\x),
\]
is a $\PP$-Dunkl-regular polynomial of degree $\deg(g)$ in $x_0,x_1,\ldots,x_n$, i.e., $CK[g]\in\F_\PP(M)$. The operator $CK$ is $\R$-linear, and if $f=CK[g]$, it holds
\begin{equation}\label{eq:poly}
f(x)=\sum_{k=0}^{\deg(f)}\frac{x_0^k}{k!}\partial_{x_0}^k f(\x),  
\end{equation}
where $\partial_{x_0}^k f(\x)=\partial_{x_0}^k f(x)_{|x_0=0}$. Moreover, $f$ is the unique $\PP$-Dunkl-regular polynomial $f$ such that $f_{|x_0=0}=g$. 
\end{theorem}
\begin{proof}
Let $f=CK[g]$. It holds
\begin{equation}\label{eq:dx0}
  \partial_{x_0}f(x)=\sum_{k=1}^{\deg(g)}(-1)^k\frac{x_0^{k-1}\DD_\PP^k g(\x)}{(k-1)!}
\end{equation}
and
\[
\DD_\PP f(x)=\sum_{k=0}^{\deg(g)-1}\frac{(-x_0)^k\DD_\PP^{k+1} g(\x)}{k!}=\sum_{h=1}^{\deg(g)}\frac{(-x_0)^{h-1}\DD_\PP^{h} g(\x)}{(h-1)!}.
\]
Therefore $D_\PP f(x)=\partial_{x_0}f(x)+\DD_\PP f(x)=0$. Thanks to Corollary \ref{cor:Pslice}, the functions $\DD_\PP^k g$ and then $f$ are $\PP$-slice. We can conclude that $f\in\F_\PP(M)$. 
From the definition of the CK-extension, we deduce immediately that $\partial_{x_0}^kf(\x)=(-1)^k\DD_\PP^k g(\x)$ for every $k$ and $\x\in M_0$, from which we get \eqref{eq:poly}.  
It remains to prove uniqueness of the extension. If $f\equiv0$ on $M_0:=\{x\in M\,|\,x_0=0\}$, then $\DD_\PP^k f\equiv0$ on $M_0$. If $f\in\F_\PP(M)$, then also the derivative $\partial_{x_0}f$ vanish identically on $M_0$. Using
\[
0= D_\PP^2 f=(\partial_{x_0}f+\DD_\PP f)(\partial_{x_0}f+\DD_\PP f)=\partial_{x_0}^2f+2\partial_{x_0}\DD_\PP f+\DD_\PP^2 f,
\]
we get also $\partial_{x_0}^2f=0$ on $M_0$. Repeating the argument we get that all the derivatives $\partial_{x_0}^kf$ vanish on $M_0$. From equation \eqref{eq:poly} we conclude that $f\equiv0$ on $M$.
\end{proof}

Note that if $g$ is a homogeneous polynomial in $x_1,\ldots, x_n$, then its CK-extension is homogeneous in $x_0,\ldots,x_n$, of the same degree as $g$. The linear CK operator can be extended to functions $g$ that are real-analytic on an open subset of $M_0=\{x\in M\,|\,x_0=0\}$. 


Given any sequence $u=(u_1,\ldots,u_\ell)$ of elements of $\Aa$ and $a\in\Aa$, we define the ordered product 
\[
[u,a]=[u_1,\ldots,u_\ell,a]:=u_1(u_2(u_3\cdots(u_{\ell-1}(u_\ell a))\cdots)).
\]

\begin{definition}
For every ${\bf d}=(d_1,\ldots,d_\ell)\in \nn^\ell$, let $\x_\PP^d=(\x_{A_1}^{d_1},\ldots,\x_{A_\ell}^{d_\ell})$. We set $P_{{\bf d},a}:=CK[[\x_\PP^d,a]]=CK[\x_{A_1}^{d_1}(\x_{A_2}^{d_2}(\cdots(\x_{A_{\ell-1}}^{d_\ell-1}(\x_{A_\ell}^{d_\ell}a))\cdots))]$. Then $P_{{\bf d},a}$ is a $\PP$-Dunkl-regular homogeneous polynomial of degree $d=d_1+\cdots+d_\ell$ in $x_0,x_1,\ldots,x_n$. When $a=1$, we also write $P_{\bf d}=P_{(d_1,\ldots,d_\ell)}$ in place of $P_{{\bf d},1}$. 
\end{definition}

\begin{proposition}\label{pro:basis}
Let $\B_\Aa=\{v_i\}_{i=0}^{\dim\Aa}$ be a real basis of $\Aa$. 
Let $d\in\nn$. The set of polynomials 
\[
\{P_{{\bf d},v}\;|\; d_1,\ldots,d_\ell\in\nn, d_1+\cdots+d_\ell=d, v\in\B_\Aa\}
\]
is a basis of the real vector space $\F_\PP(M)\cap \Aa_d[x_0,\ldots,x_n]$ of $\PP$-Dunkl-regular homogeneous polynomials of degree $d$ on $M$. Therefore, this space has real dimension $\binom{\ell+d-1}d\dim\Aa$. 
\end{proposition}
\begin{proof}
Given a subset $A$ of $[n]$, let $g\in\ker\mathscr S_A\cap \Aa_d[x_1,\ldots,x_n]$. Decompose $g$ as $g=\sum_{i=0}^d g_{i}$, where $g_{i}$ is homogeneous of degree $i$ w.r.t.\ the set of variables $\{x_j\;|\; j\in A\}$. 
Since $\SA$ is a 0-degree operator that preserve homogeneity w.r.t.\ $x_j$, $j\in A$, it holds $0=\SA g=\sum_{i=0}^d \SA g_{i}$ if and only if $\SA g_{i}=0$ for every $i=0,\ldots,d$. 
Using \eqref{eq:SA}, the equality $\SA g_{i}=0$ is equivalent to
\[
\Eu_A(g_{i})=i g_{i}=-\x_A\,\DD_A g_{i}.
\]
Since $\SA(\DD_A g_{i})=-\DD_A(\SA g_{i})=0$ (\cite[Remark 34]{BinosiPerotti2025}) and $\DD_\B g_{i}$ is $(i-1)$-homogeneous  w.r.t.\ $\{x_j\;|\; j\in A\}$ (see \cite[Lemma\ 2.9]{Rosler}), we also have 
\[
(i -1)\DD_A g_{i}=-\x_A\, \DD_A^2 g_{i}, 
\]
and then $i(i-1)g_{i}=(-\x_A)^2\DD_A^2 g_{i}$. Iterating this computation, we finally get
\[
g_{i}=(i!)^{-1}(-\x_A)^i\DD_A^i g_{i}=\x_A^i \tilde g_i
\]
for a polynomial function $\tilde g_i $ not depending on $x_i$, $i\in A$. Therefore $g=\sum_{i=0}^d \x_A^i \tilde g_i$. 

If $g\in\ker\mathscr S_\PP\cap \Aa_d[x_1,\ldots,x_n]$, then $g\in\cap_{j=1}^\ell\ker\mathscr S_{A_j}$. By the previous argument, we obtain
\[
g=\sum_{i=0}^d \x_{A_1}^i \tilde g_{i,1},
\]
with $\tilde g_{i,1}$ a polynomial function in the set of variables $\{x_j\;|\; j\in[n]\setminus A_1\}$. Since \cite[Proposition 50]{BinosiPerotti2025} implies that $0=\mathscr S_{A_2}g=\sum_{i=0}^d \x_{A_1}^i \mathscr S_{A_2}(\tilde g_{i,1})$, it must be $\mathscr S_{A_2}(\tilde g_{i,1})=0$ for every $i$. Repeating the argument above, we finally get the decomposition
\[
g=\sum_{d_1+\cdots +d_\ell=d} \x_{A_1}^{d_1}(\x_{A_2}^{d_2}(\cdots(\x_{A_{\ell-1}}^{d_\ell-1}(\x_{A_\ell}^{d_\ell}a_{{\bf d}}))\cdots))=\sum_{d_1+\cdots +d_\ell=d}[\x_\PP^d,a_{{\bf d}}], 
\]
with $a_{{\bf d}}=\sum_{i=0}^{\dim\Aa}a_{{\bf d},i}v_i\in\Aa$ and $a_{{\bf d},i}\in\R$. Therefore $CK[g]=\sum_{i=0}^{\dim\Aa}\sum_{d_1+\cdots +d_\ell=d}a_{{\bf d},i} P_{{\bf d},v_i}$. 

If $f\in\F_\PP(M)\cap \Aa_d[x_0,\ldots,x_n]$, from Theorem \ref{teo:CK} we get that $f=CK[f_{|\{x_0=0\}}]$. Since the restriction $f_{|\{x_0=0\}}$ is a $\PP$-slice homogenous polynomial in $x_1,\ldots,x_n$, we conclude that $f$ is a $\R$-linear combination of the polynomials $P_{{\bf d},v_i}$.  
\end{proof}

Let $\epsilon_j=(0,\ldots,1,\ldots,0)\in\nn^\ell$, with a 1 in $j$-th position. 
The functions $P_{\epsilon_j}=CK[\x_{A_j}]=x_0+\x_{A_j}=x_{A_j}$, $j=1,\ldots,\ell$, form a basis for linear polynomials in $\F_\PP(M)$. More generally, it holds $P_{m\epsilon_j}=CK[\x_{A_j}^m]=(x_0+\x_{A_j})^m={(x_{A_j})}^m\in\F_\PP(M)$, $j=1,\ldots,\ell$ and $m\in\nn$. From Theorem \ref{teo:CK} we get that, for every $k=1,\ldots m$, 
\begin{equation}\label{eq:DkP}
\DD_\PP^k((\x_{A_j})^m)=(-1)^k \frac{m!}{(m-k)!}(\x_{A_j})^{m-k}. 
\end{equation}
Since $0=D_\PP(P_{m\epsilon_j})=(\partial_{x_0}+\DD_\PP)P_{m\epsilon_j}$, we get that
\[
\tfrac12D_\PP^c(P_{m\epsilon_j})=\partial_{x_0}P_{m\epsilon_j}= m\,P_{(m-1)\epsilon_j}
\]
for every $m\ge1$, where $P_{0\epsilon_j}=CK[1]=1$. Then for every $j=1,\ldots,\ell$, the family $\{P_{m\epsilon_j}\}_{m\in\nn}\subseteq\F_\PP(M)$ is an \emph{Appell sequence} w.r.t.\ the operator $\tfrac12D_\PP^c=\tfrac12(\partial_{x_0}-\DD_\PP)$. 

Observe that in general $P_{\bf d}=CK[[\x_\PP^d,1]]$ is different from  the product $[x_\PP^d,1]$. We have 
\[
\tfrac12D_\PP^c(P_{{\bf d}})=\partial_{x_0}P_{{\bf d}}=-\DD_\PP P_{\bf d}=-CK[\DD_\PP(\x_{A_1}^{d_1}(\x_{A_2}^{d_2}(\cdots(\x_{A_{\ell-1}}^{d_\ell-1}\x_{A_\ell}^{d_\ell}))\cdots))].
\]

\subsubsection*{The associative case}
Assume that $\Aa$ is associative. 

\begin{definition}
For every $(d_1,\ldots,d_\ell)\in \nn^\ell$, we set $P_{(d_1,\ldots,d_\ell)}:=CK[\x_{A_1}^{d_1}\cdots\x_{A_\ell}^{d_\ell}]$. Then $P_{(d_1,\ldots,d_\ell)}$ is a $\PP$-Dunkl-regular homogeneous polynomial of degree $d=d_1+\cdots+d_\ell$ in $x_0,x_1,\ldots,x_n$. It holds $P_{{\bf d},v_i}=P_{(d_1,\ldots,d_\ell)}v_i$ for every $i=0,\ldots,\dim\Aa$.
\end{definition}

\begin{corollary}\label{cor:basis_associative}
Let $d\in\nn$. The set of polynomials 
\[\{P_{(d_1,\ldots,d_\ell)}\;|\; d_1,\ldots,d_\ell\in\nn, d_1+\cdots+d_\ell=d\}
\]
spans independently the right $\Aa$-module $\F_\PP(M)\cap \Aa_d[x_0,\ldots,x_n]$ of $\PP$-Dunkl-regular homogeneous polynomials of degree $d$ on $M$. Therefore, this module has dimension $\binom{\ell+d-1}d$ over $\Aa$. 
\end{corollary}

\begin{example}\label{ex:P122}
Let $\Aa=\R_6$ and consider the eight-dimensional hypercomplex subspace $M=\langle1,e_1,e_2,e_3,e_4,e_5,e_6,e_{123456}\rangle$. Let $\PP=\{A_1,A_2,A_3\}$ be the partition of the set $[7]$ with elements $A_1=\{1\}$, $A_2=\{2,3,4\}$, $A_3=\{5,6,7\}$. Let $\x_{A_1}=x_1e_1$,  $\x_{A_2}=x_2e_2+x_3e_3+x_4e_4$ and $\x_{A_3}=x_5e_5+x_6e_6+x_{123456}e_{123456}$. 
From Theorem \ref{teo:CK}, we obtain the $\PP$-Dunkl-regular polynomial $P_{(1,2,2)}=CK[\x_{A_1}\x_{A_2}^2\x_{A_3}^2]$:
\begin{align*}
P_{(1,2,2)}&=
\tfrac{1}{15}x_0^5+
\tfrac{1}3x_0^4 \x_{A_1}-
\tfrac{1}3x_0^3\left(2\x_{A_1}\big(\x_{A_2}+\x_{A_3}\big)+\big(\|\x_{A_2}\|^2+\|\x_{A_3}\|^2\big)\right)-
x_0^2 \x_{A_1}\big(\|\x_{A_2}\|^2+\|\x_{A_3}\|^2\big)\\
&\quad +x_0\left(2\x_{A_1}\big(\x_{A_2}\|\x_{A_3}\|^2+\x_{A_3}\|\x_{A_2}\|^2\big)+\|\x_{A_2}\|^2\|\x_{A_3}\|^2 \right)
+\x_{A_1} \|\x_{A_2}\|^2\|\x_{A_3}\|^2 .
\end{align*}
\end{example}

\begin{remark}
The polynomial $P_{(1,2,2)}$ coincides with the product $\mathcal T_{(1,2,2)}e_1$, where $\mathcal T_{(1,2,2)}$ is the $T$-regular polynomial defined following \cite[Definition 5.6]{GhiloniStoppato_arXiv24}. The polynomials $\mathcal T_{(d_1,\ldots,d_\ell)}$ defined recursively in \cite{GhiloniStoppato_arXiv24} in the associative case, 
form another basis of the $\Aa$-submodule $\F_\PP(M)\cap \Aa_d[x_0,\ldots,x_n]$ of $\F_\PP(M)$, different from our basis $\{P_{(d_1,\ldots,d_\ell)}\}$ obtained by CK-extension of the products $\x_{A_1}^{d_1}\cdots\x_{A_\ell}^{d_\ell}$. 
\end{remark}

\section{General Fueter Theorem and Fueter trees}\label{sec:GFT}

If $\PP=\{A_1,\ldots,A_\ell\}$ and ${\bf k}=(k_1,\ldots,k_n)$ is a sequence of $\PP$-admissible Dunkl multiplicities, then the sum $\kappa$ of the Dunkl multiplicities will be called the \emph{Dunkl weight} of the functions in $\F_\PP(\OO)$. It holds $\kappa=\sum_{j=1}^\ell(1-|A_j|)/2=(\ell-n)/2\le0$. 

\begin{theorem}\label{teo:DeltaM}
Let $\OO\subseteq M$ be open and 
$\PP$-circular. Assume that every set $A_j\in\PP$ has cardinality $|A_j|\ne2$. 
Let $f\in\F_\PP(\OO)$ 
be a Dunkl-regular function of Dunkl weight $\kappa=(\ell-n)/2<0$, where $\ell=|\PP|$. Then its Laplacian $\Delta_M f$ is a sum of at most $\ell$ Dunkl-regular functions of Dunkl weight $\kappa+1$. Moreover, $\Delta_M f$ is a $\PP$-slice function on $\OO$. More precisely, 
\[
\Delta_M f=\sum_{j\,;\,|A_j|>2}g_j,
\]
where $g_j\in\F_{\PP_j}(\OO)$, with $\PP_j=(\PP\setminus A_j)\cup \{\{i_1\},\{i_2\},\{i_3,\ldots,i_{|A_j|}\}\}$ if $A_j=\{i_1,i_2,\ldots,i_{|A_j|}\}$.  
\end{theorem}

\begin{remark}\label{rem:Pj}
The statement of Theorem \ref{teo:DeltaM} in particular implies that the functions $g_j$ belong to the space $\F_{\PP_j}(\OO)$ for any partition $\PP_j$ obtained by $\PP$ deleting from $A_j$ two elements and adding the corresponding two singletons.
\end{remark}

Before proving the theorem, we gives some immediate corollaries of this result.
In the following, we call $\PP$ an \emph{odd partition} if every set $A_j\in\PP$ has odd cardinality. This is equivalent to have $\kappa_{A_j}:=\sum_{i\in A_j}k_i\in\zz$ for all $j=1,\ldots,\ell$.

\begin{corollary}[General Fueter Theorem]\label{cor:GFT}
Let $\OO\subseteq M$ be open and $\PP$-circular. 
Assume that $\PP$ is an odd partition.  Let $f\in\F_\PP(\OO)$ be a Dunkl-regular function of 
Dunkl weight $\kappa<0$. 
Then $\kappa$ is a negative integer and the iterated Laplacian $\Delta_M^{|\kappa|}f=(\Delta_M)^{\frac{n-\ell}2}f$ is a $\PP$-slice monogenic function on $\OO$, namely, it belongs to $\ker\mathscr S_\PP$ and
\[
\dM{(\Delta_M)}^{\frac{n-\ell}2}f=0. 
\]
\hfill\qed
\end{corollary}

\begin{corollary}\label{cor:polyharmonic}
Let $\OO\subseteq M$ be open and $\PP$-circular. 
Assume that $\PP$ is an odd partition. Then every Dunkl-regular function $f\in\F_\PP(\OO)$ is a polyharmonic function of order $|\kappa|+1=\frac{n-\ell}2+1$  on $\OO$, namely,
\[
\Delta_M^{|\kappa|+1}f=(\Delta_M)^{\frac{n-\ell+2}2}f=0. 
\]
\hfill\qed
\end{corollary}

Corollaries \ref{cor:GFT} and \ref{cor:polyharmonic} improve known Fueter-type theorems also in the slice regular functions case, i.e., with $\PP=\{[n]\}$, on non-associative hypercomplex subspaces (see \cite[Theorem 27]{CRoperators}, where an additional condition on the functions was assumed). 

Observe that the version of Fueter Theorem given by Corollary \ref{cor:GFT} for the class of $T$-regular functions (see Example \ref{ex:FP} (iv))  is different from the Fueter-type Theorem proved in \cite{GhiloniStoppatoFueterSce}, where the power of the Euclidean Laplacian is replaced by the composition of other differential operators of the second order.

If $\PP$ is not an odd partition, at least 
one $\kappa_{A_j}$ belongs to the semi-integers $\tfrac12\zz$. 
Theorem \ref{teo:DeltaM} can be applied until one of the partitions $\PP_j$ of the statement contains one 
element of cardinality two. 
In particular, when $\ell=1$ and $n$ is even, one can apply the Laplacian $(n-2)/2$ times.

\begin{corollary}\label{cor:GFTnonodd}
Let $n=\dim M-1>2$ be even. Let $\OO\subseteq M$ be open and axially symmetric.  If $f\in\F_{[n]}(\OO)=\sr(\OO)$ is a slice-regular function, then the iterated Laplacian $\Delta_M^{\frac{n-2}2}f$ is a slice function on $\OO$ in the space $\F_{\PP'}(\OO)$, where $\PP'=\{\{1\},\{2\},\ldots,\{n-1,n\}\}$.  
\hfill\qed
\end{corollary}

\begin{example}
Let $M=\R^5\subseteq\R_4$ be the paravector space. Since $x^4\in\sr(M)$, then 
\[
\Delta_M (x^4)=-12(3x_0^2-\|\IM(x)\|^2+2x_0\IM(x))
\]
is 
a  Dunkl-regular slice function in the space $\F_{\{1\},\{2\},\{3,4\}}(M)$, in the kernel of the operator (see Remark \ref{rem:DM_delta1})
\[
D_{\{1\},\{2\},\{3,4\}}=\dM-\tfrac12e_4\delta_1^4=\dM-\tfrac12e_4\frac{1-r_4}{x_4}.
\]
The function $\Delta_M(x^4)$ is then Dunkl harmonic on $M$ but not harmonic, since $\Delta_M^2(x^4)=24$.  
\end{example}

\begin{proof}[Proof of Theorem \ref{teo:DeltaM}]
Let $\B=\{1,v_1,\ldots,v_n\}$ be a fixed hypercomplex basis of $M$, with real coordinates $x_0,x_1,\ldots,x_n$.  Let $\delta_1,\delta_2$ be the operators defined in \eqref{eq:delta12}.

Let $f\in\F_\PP(\OO)$ and let $\alpha=(\alpha_1,\ldots,\alpha_\ell)$, with $\alpha_j$ an element of $A_j$ for every $j$. Let $(k_1,\ldots,k_n)$ be the multiplicities defined as in \eqref{eq:canonical}. Let $A_j=\{i_1=\alpha_j,i_2,\ldots,i_{|A_j|}\}$. 
Then $k_{i_1}=k_{\alpha_j}=0$ and $k_{i_2}=\cdots k_{i_{|A_j|}}=-1/2$. 

In view of Proposition \ref{pro:DeltaM}, it is sufficient to prove that for every $j\in\{1,\ldots,\ell\}$ with $|A_j|>2$, the function $\delta_2^{\alpha_j}f$ belongs to the space $\F_{\PP_j}(\OO)$, where $\PP_j=(\PP\setminus A_j)\cup \{\{\alpha_j\},\{i_2\},\{i_3,\ldots,i_{|A_j|}\}\}$ is a new partition of $[n]$ with $\ell+2$ elements.

In order to prove that $\delta_2^{\alpha_j}f\in\ker D_{\PP_j}$, we show that $D_{\PP_j}(\delta_2^{\alpha_j}f)=\delta_2^{\alpha_j}(D_{\PP}f)$. Let $(k'_1,\ldots,k'_n)$ be the $\PP_j$-admissible sequence obtained from $(k_1,\ldots,k_n)$ replacing the elements $k_{i_2}=-1/2$ and $k_{i_3}=-1/2$ with two zeros: $k'_{i_2}=k'_{i_3}=0$, $k'_i=k_i=-1/2$ for every $i\ne i_2,i_3$.  From Remark \ref{rem:DM_delta1} we can write
\[
D_{\PP_j}=\dM-\tfrac12\sum_{\substack{i=1\\i\not\in\alpha\cup\{i_2,i_3\}}}^n v_i\delta_1^i.
\]
Therefore
\[
D_\PP-D_{\PP_j}=-\tfrac12\left(v_{i_2}\delta_1^{i_2}+v_{i_3}\delta_1^{i_3}\right). 
\]
Since $[\delta_1^i,\delta_2^j]=0=[\partial_{x_i},\delta_2^j]$ for every $j\ne i$ and $\alpha_j=i_1$, we can write
\[
D_{\PP_j}(\delta_2^{\alpha_j}f)=\bigg(\dM-\tfrac12\sum_{\substack{i=1\\i\not\in\alpha\cup\{i_2,i_3\}}}^n  v_i\delta_1^i\bigg)(\delta_2^{\alpha_j}f)=
\delta_2^{\alpha_j}\bigg(\dM f-\tfrac12\sum_{\substack{i=1\\i\not\in\alpha\cup\{i_2,i_3\}}}^n  v_i\delta_1^if\bigg)+v_{i_1}[\partial_{x_{i_1}},\delta_2^{i_1}]f.
\]
From Proposition \ref{pro:delta1_2slice} it holds $v_{i_1}\delta_1^{i_1}f=v_{i_2}\delta_1^{i_2}f=v_{i_3}\delta_1^{i_3}f$, since $i_1,i_2,i_3\in A_j$, and from Proposition \ref{pro:delta12} we get
\[
[\partial_{x_{i_1}},\delta_2^{i_1}]f=-\delta_2^{i_1}(\delta_1^{i_1}f).
\]
Then
\[
D_{\PP_j}(\delta_2^{\alpha_j}f)=\delta_2^{\alpha_j}\bigg(\dM f-\tfrac12\sum_{\substack{i=1\\i\not\in\alpha\cup\{i_2,i_3\}}}^n  v_i\delta_1^if\bigg)-\tfrac12 \delta_2^{\alpha_j}(v_{i_2}\delta_1^{i_2}f+v_{i_3}\delta_1^{i_3}f)
=\delta_2^{\alpha_j}(D_\PP f).
\]
From Corollary \ref{cor:delta1_2Pslice}, $\delta_2^{\alpha_j} f\in\ker\mathscr S_\PP$, and then it also belongs to $\ker\mathscr S_{\PP_j}$ (see \cite[Proposition 49]{BinosiPerotti2025}). It follows that $\delta_2^{\alpha_j} f$ is a $\PP$-slice function in the space $\F_{\PP_j}(\OO)$. 

We conclude setting $g_j:=2^{-1}(1-|A_j|)\delta_2^{\alpha_j} f$ for any $j=1,\ldots,\ell$. The functions $g_j$ have Dunkl weight $2^{-1}\sum_{i\ne j}^\ell(1-|A_i|)+2^{-1}(1-(|A_j|-2))=\kappa+1$, as required.
\end{proof}

\begin{example}
Let $M$ be the eight-dimensional hypercomplex subspace 
\[
M=\langle1,e_1,e_2,e_3,e_4,e_5,e_6,e_{123456}\rangle
\]
of $\R_6$. Let $\PP=\{A_1,A_2,A_3\}$, with $A_1=\{1\}$, $A_2=\{2,3,4\}$, $A_3=\{5,6,7\}$. Consider the polynomial $f=P_{(1,2,2)}\in\F_\PP(M)$ of Example \ref{ex:P122}:
\begin{align*}
f&=\tfrac{1}{15}x_0^5+
\tfrac{1}3x_0^4 \x_{A_1}-
\tfrac{1}3x_0^3\left(2\x_{A_1}\big(\x_{A_2}+\x_{A_3}\big)+\big(\|\x_{A_2}\|^2+\|\x_{A_3}\|^2\big)\right)-
x_0^2 \x_{A_1}\big(\|\x_{A_2}\|^2+\|\x_{A_3}\|^2\big)\\
&\quad +x_0\left(2\x_{A_1}\big(\x_{A_2}\|\x_{A_3}\|^2+\x_{A_3}\|\x_{A_2}\|^2\big)+\|\x_{A_2}\|^2\|\x_{A_3}\|^2 \right)
+\x_{A_1} \|\x_{A_2}\|^2\|\x_{A_3}\|^2 .
\end{align*}
Through the operators $\delta_1^{i}$ we compute the spherical derivatives $f'_{s,A_j}$ for every $j$ with $|A_j|>1$:
\[
f'_{s,A_2}=\frac23x_0^3x_1e_1-2x_0x_1e_1\|\x_{A_3}\|^2,\quad f'_{s,A_3}=\frac23x_0^3x_1e_1-2x_0x_1e_1\|\x_{A_2}\|^2,
\]
and then the polynomial $\dM f$:
\[
\dM f=-2f'_{s,A_2}-2f'_{s,A_3}=-\frac83x_0^3x_1e_1+4x_0x_1e_1\big(\|\x_{A_2}\|^2+\|\x_{A_3}\|^2\big).
\]
From this we get that $\Delta_M(\dM f)=32x_0x_1e_1$ and then $\Delta_M^2(\dM f)=\dM(\Delta_M^2f)=0$, according to Corollary \ref{cor:GFT}. In order to compute the monogenic function $\Delta_M^2f$, we arrive at the same result through a longer way, firstly computing the Laplacian $\Delta_Mf$ and then its square $\Delta_M^2 f$ using the operators $\delta_2^{j}$ and Theorem \ref{teo:DeltaM}. The polynomial $f$ is a Dunkl-regular function of weight $-2$. 
We construct two Dunkl-regular functions of weight $-1$: 
\begin{align*}
g_2&=-\delta_2^{2}f=-\frac43 x_0^3-4x_0^2 x_1e_1+4x_0\big(2x_1e_1\x_{A_3}+\|\x_{A_3}\|^2\big)+4x_1e_1\|\x_{A_3}\|^2\in\F_{\PP_2}(M),\\
g_3&=-\delta_2^{5}f=-\frac43 x_0^3-4x_0^2 x_1e_1+4x_0\big(2x_1e_1\x_{A_2}+\|\x_{A_2}\|^2\big)+4x_1e_1\|\x_{A_2}\|^2\in\F_{\PP_3}(M),
\end{align*}
where $\PP_2=\{\{1\},\{2\},\{3\},\{4\},\{5,6,7\}\}$, $\PP_3=\{\{1\},\{2,3,4\},\{5\},\{6\},\{7\}\}$. It follows that
\begin{align*}
\Delta_M f&=g_2+g_3\\
&=-\frac83 x_0^3-8x_0^2 x_1e_1+4x_0\left(2x_1e_1\big(\x_{A_2}+\x_{A_3}\big)+\big(\|\x_{A_2}\|^2+\|\x_{A_3}\|^2\big)\right)+4x_1e_1\big(\|\x_{A_2}\|^2+\|\x_{A_3}\|^2\big).
\end{align*}
Repeating the procedure provided by Theorem \ref{teo:DeltaM} on $g_2$ and $g_3$, we obtain the monogenic polynomials (Dunkl-regular functions of weight 0)
\[
\Delta_M g_2=-\delta_2^{5}g_2=16(x_0+x_1 e_1), \quad \Delta_M g_3=-\delta_2^{2}g_3=16(x_0+x_1 e_1).
\]
Therefore $\Delta_M^2 f=\Delta_M g_2+\Delta_M g_3=32(x_0+x_1 e_1)$ is a monogenic function on $M$, in accordance with Corollary \ref{cor:GFT}.
\end{example}

\begin{example}
Let $M=\oo$ be the algebra of octonions. Since $x^3\in\sr(\oo)$, its Laplacian $\Delta_8(x^3)$ is a slice polynomial in the space $\F_\PP(\OO)$, where $\PP=\{\{1\},\{2\},\{3,4,5,6,7\}\}$, in accordance with Theorem \ref{teo:DeltaM}. As observed in \cite[\S4.3]{BinosiPerotti2025}, the slice functions in this space are slice Fueter-regular functions \cite{JinRenSabadini2020,SliceFueterRegular}. 
A direct computation shows that
\[
\Delta_8(x^3)=-12\left(3x_0+\IM(x)\right).
\]
In particular, $\Delta_8(x^3)$ is a Dunkl monogenic function of weight $-2$, in the kernel of $D_\PP=\dibar_\oo-\tfrac12\sum_{i=4}^7\frac{v_i}{x_i}(1-r_i)$.
\end{example}

If every set $A_j\in\PP$ has cardinality $|A_j|\ne2$, we can define, for $j=1,\ldots,\ell$, the operators $\Tau_j:\F_\PP(\OO)\to\F_{\PP_j}(\OO)$  as
\[
\Tau_j(f)=\tfrac{1-|A_j|}2\delta_2^{\alpha_j} f,
\]
where $\alpha_j$ is any element of $A_j$. We also set $\Tau=(\Tau_1,\ldots,\Tau_\ell):\F_\PP(\OO)\to(\F_{\PP_1}(\OO),\ldots,\F_{\PP_\ell}(\OO))$. 
Theorem \ref{teo:DeltaM} proves that for every $f\in\F_\PP(\OO)$, it holds
\[
\Delta_M f=\sum_{j=1}^\ell\Tau_j(f).
\]
If $\PP$ is an odd partition and $\ell<n$, we can iterate the operator $\Tau$ and obtain a directed rooted tree, the \emph{Fueter tree} of the space $\F_\PP(\OO)$. It is a tree with root $\F_\PP(\OO)$, nodes some spaces $\F_{\PP'}(\OO)$ with $\PP'$ a refinement of $\PP$, and leaves all equal to the space of monogenic functions $\M(\OO)$. The Fueter tree has height $|\kappa|=(n-\ell)/2$, where $\kappa<0$ is the Dunkl weight of the elements of $\F_\PP(\OO)$. 
If $|A_j|=1$, $\Tau_j\equiv0$ and we can omit the corresponding edge and child node. Therefore the Fueter tree of $\F_\PP(\OO)$ is a $\ell'$-ary tree, where $\ell'$ denotes the number of elements $A_j\in\PP$ with cardinality greater than one, i.e., there are at most $\ell'$ non-zero operators $\Tau_j$ with origin any node of the tree (see Figure \ref{figureTree}).

Observe that for every $f\in\F_\PP(\OO)$ and $i=1,\ldots,|\kappa|$, the functions obtained iterating $\Tau$ and the iterated Laplacian $\Delta^i_M f$ are $\PP$-slice on $\OO$. In particular, $\Delta^{\frac{n-\ell}2}_{M}f\in\M(\OO)\cap\ker\mathscr S_\PP$, as stated in Corollary \ref{cor:GFT}.  
Then every Fueter tree provides a version of the Fueter Theorem. In particular we subsume all the already known Fueter-type Theorems. 

\begin{figure}[h]
\tikzset{every picture/.style={line width=0.75pt}} 
\begin{center}
\begin{tikzpicture}[x=0.75pt,y=0.75pt,yscale=-1,xscale=0.9]

\draw    (301.1,30.28) -- (147.77,131.18) ;
\draw [shift={(146.1,132.28)}, rotate = 326.65] [color={rgb, 255:red, 0; green, 0; blue, 0 }  ][line width=0.75]    (10.93,-3.29) .. controls (6.95,-1.4) and (3.31,-0.3) .. (0,0) .. controls (3.31,0.3) and (6.95,1.4) .. (10.93,3.29)   ;
\draw    (301.1,30.28) -- (455.31,127.93) ;
\draw [shift={(457,129)}, rotate = 212.34] [color={rgb, 255:red, 0; green, 0; blue, 0 }  ][line width=0.75]    (10.93,-3.29) .. controls (6.95,-1.4) and (3.31,-0.3) .. (0,0) .. controls (3.31,0.3) and (6.95,1.4) .. (10.93,3.29)   ;
\draw    (301.1,30.28) -- (258.29,132.71) ;
\draw [shift={(257.52,134.55)}, rotate = 292.68] [color={rgb, 255:red, 0; green, 0; blue, 0 }  ][line width=0.75]    (10.93,-3.29) .. controls (6.95,-1.4) and (3.31,-0.3) .. (0,0) .. controls (3.31,0.3) and (6.95,1.4) .. (10.93,3.29)   ;
\draw    (146.1,132.28) -- (85.99,219.57) ;
\draw [shift={(84.86,221.22)}, rotate = 304.55] [color={rgb, 255:red, 0; green, 0; blue, 0 }  ][line width=0.75]    (10.93,-3.29) .. controls (6.95,-1.4) and (3.31,-0.3) .. (0,0) .. controls (3.31,0.3) and (6.95,1.4) .. (10.93,3.29)   ;
\draw    (146.1,132.28) -- (182.08,218.04) ;
\draw [shift={(182.86,219.88)}, rotate = 247.24] [color={rgb, 255:red, 0; green, 0; blue, 0 }  ][line width=0.75]    (10.93,-3.29) .. controls (6.95,-1.4) and (3.31,-0.3) .. (0,0) .. controls (3.31,0.3) and (6.95,1.4) .. (10.93,3.29)   ;
\draw    (257.52,134.55) -- (228.82,221.33) ;
\draw [shift={(228.19,223.23)}, rotate = 288.3] [color={rgb, 255:red, 0; green, 0; blue, 0 }  ][line width=0.75]    (10.93,-3.29) .. controls (6.95,-1.4) and (3.31,-0.3) .. (0,0) .. controls (3.31,0.3) and (6.95,1.4) .. (10.93,3.29)   ;
\draw    (257.52,134.55) -- (296.06,220.3) ;
\draw [shift={(296.88,222.13)}, rotate = 245.8] [color={rgb, 255:red, 0; green, 0; blue, 0 }  ][line width=0.75]    (10.93,-3.29) .. controls (6.95,-1.4) and (3.31,-0.3) .. (0,0) .. controls (3.31,0.3) and (6.95,1.4) .. (10.93,3.29)   ;
\draw    (457,129) -- (434.67,220.07) ;
\draw [shift={(434.19,222.02)}, rotate = 283.78] [color={rgb, 255:red, 0; green, 0; blue, 0 }  ][line width=0.75]    (10.93,-3.29) .. controls (6.95,-1.4) and (3.31,-0.3) .. (0,0) .. controls (3.31,0.3) and (6.95,1.4) .. (10.93,3.29)   ;
\draw    (457,129) -- (490.8,216.69) ;
\draw [shift={(491.52,218.55)}, rotate = 248.92] [color={rgb, 255:red, 0; green, 0; blue, 0 }  ][line width=0.75]    (10.93,-3.29) .. controls (6.95,-1.4) and (3.31,-0.3) .. (0,0) .. controls (3.31,0.3) and (6.95,1.4) .. (10.93,3.29)   ;
\draw  [dash pattern={on 4.5pt off 4.5pt}]  (301.1,30.28) -- (351.3,132.1) ;
\draw [shift={(352.19,133.9)}, rotate = 243.76] [color={rgb, 255:red, 0; green, 0; blue, 0 }  ][line width=0.75]    (10.93,-3.29) .. controls (6.95,-1.4) and (3.31,-0.3) .. (0,0) .. controls (3.31,0.3) and (6.95,1.4) .. (10.93,3.29)   ;
\draw  [dash pattern={on 4.5pt off 4.5pt}]  (84.86,221.22) -- (84.2,326.68) ;
\draw [shift={(84.19,328.68)}, rotate = 270.36] [color={rgb, 255:red, 0; green, 0; blue, 0 }  ][line width=0.75]    (10.93,-3.29) .. controls (6.95,-1.4) and (3.31,-0.3) .. (0,0) .. controls (3.31,0.3) and (6.95,1.4) .. (10.93,3.29)   ;
\draw  [dash pattern={on 4.5pt off 4.5pt}]  (182.86,219.88) -- (182.86,326.68) ;
\draw [shift={(182.86,328.68)}, rotate = 270] [color={rgb, 255:red, 0; green, 0; blue, 0 }  ][line width=0.75]    (10.93,-3.29) .. controls (6.95,-1.4) and (3.31,-0.3) .. (0,0) .. controls (3.31,0.3) and (6.95,1.4) .. (10.93,3.29)   ;
\draw  [dash pattern={on 4.5pt off 4.5pt}]  (228.19,223.23) -- (228.05,327.29) ;
\draw [shift={(228.05,329.29)}, rotate = 270.08] [color={rgb, 255:red, 0; green, 0; blue, 0 }  ][line width=0.75]    (10.93,-3.29) .. controls (6.95,-1.4) and (3.31,-0.3) .. (0,0) .. controls (3.31,0.3) and (6.95,1.4) .. (10.93,3.29)   ;
\draw  [dash pattern={on 4.5pt off 4.5pt}]  (296.88,222.13) -- (298.84,326.92) ;
\draw [shift={(298.88,328.91)}, rotate = 268.93] [color={rgb, 255:red, 0; green, 0; blue, 0 }  ][line width=0.75]    (10.93,-3.29) .. controls (6.95,-1.4) and (3.31,-0.3) .. (0,0) .. controls (3.31,0.3) and (6.95,1.4) .. (10.93,3.29)   ;
\draw  [dash pattern={on 4.5pt off 4.5pt}]  (352.19,133.9) -- (333.94,220.06) ;
\draw [shift={(333.52,222.02)}, rotate = 281.96] [color={rgb, 255:red, 0; green, 0; blue, 0 }  ][line width=0.75]    (10.93,-3.29) .. controls (6.95,-1.4) and (3.31,-0.3) .. (0,0) .. controls (3.31,0.3) and (6.95,1.4) .. (10.93,3.29)   ;
\draw  [dash pattern={on 4.5pt off 4.5pt}]  (352.19,133.9) -- (360.66,220.03) ;
\draw [shift={(360.86,222.02)}, rotate = 264.38] [color={rgb, 255:red, 0; green, 0; blue, 0 }  ][line width=0.75]    (10.93,-3.29) .. controls (6.95,-1.4) and (3.31,-0.3) .. (0,0) .. controls (3.31,0.3) and (6.95,1.4) .. (10.93,3.29)   ;
\draw  [dash pattern={on 4.5pt off 4.5pt}]  (352.19,133.9) -- (388.09,220.17) ;
\draw [shift={(388.86,222.02)}, rotate = 247.41] [color={rgb, 255:red, 0; green, 0; blue, 0 }  ][line width=0.75]    (10.93,-3.29) .. controls (6.95,-1.4) and (3.31,-0.3) .. (0,0) .. controls (3.31,0.3) and (6.95,1.4) .. (10.93,3.29)   ;
\draw  [dash pattern={on 4.5pt off 4.5pt}]  (360.86,222.02) -- (362.16,326.02) ;
\draw [shift={(362.19,328.02)}, rotate = 269.28] [color={rgb, 255:red, 0; green, 0; blue, 0 }  ][line width=0.75]    (10.93,-3.29) .. controls (6.95,-1.4) and (3.31,-0.3) .. (0,0) .. controls (3.31,0.3) and (6.95,1.4) .. (10.93,3.29)   ;
\draw  [dash pattern={on 4.5pt off 4.5pt}]  (388.86,222.02) -- (389.98,328) ;
\draw [shift={(390,330)}, rotate = 269.39] [color={rgb, 255:red, 0; green, 0; blue, 0 }  ][line width=0.75]    (10.93,-3.29) .. controls (6.95,-1.4) and (3.31,-0.3) .. (0,0) .. controls (3.31,0.3) and (6.95,1.4) .. (10.93,3.29)   ;
\draw  [dash pattern={on 4.5pt off 4.5pt}]  (146.1,132.28) -- (129.43,219.83) ;
\draw [shift={(129.05,221.79)}, rotate = 280.79] [color={rgb, 255:red, 0; green, 0; blue, 0 }  ][line width=0.75]    (10.93,-3.29) .. controls (6.95,-1.4) and (3.31,-0.3) .. (0,0) .. controls (3.31,0.3) and (6.95,1.4) .. (10.93,3.29)   ;
\draw  [dash pattern={on 4.5pt off 4.5pt}]  (257.52,134.55) -- (259.02,219.29) ;
\draw [shift={(259.05,221.29)}, rotate = 268.99] [color={rgb, 255:red, 0; green, 0; blue, 0 }  ][line width=0.75]    (10.93,-3.29) .. controls (6.95,-1.4) and (3.31,-0.3) .. (0,0) .. controls (3.31,0.3) and (6.95,1.4) .. (10.93,3.29)   ;
\draw  [dash pattern={on 4.5pt off 4.5pt}]  (457,129) -- (461.94,219.79) ;
\draw [shift={(462.05,221.79)}, rotate = 266.88] [color={rgb, 255:red, 0; green, 0; blue, 0 }  ][line width=0.75]    (10.93,-3.29) .. controls (6.95,-1.4) and (3.31,-0.3) .. (0,0) .. controls (3.31,0.3) and (6.95,1.4) .. (10.93,3.29)   ;
\draw  [dash pattern={on 4.5pt off 4.5pt}]  (129.05,221.79) -- (128.4,327.26) ;
\draw [shift={(128.38,329.26)}, rotate = 270.36] [color={rgb, 255:red, 0; green, 0; blue, 0 }  ][line width=0.75]    (10.93,-3.29) .. controls (6.95,-1.4) and (3.31,-0.3) .. (0,0) .. controls (3.31,0.3) and (6.95,1.4) .. (10.93,3.29)   ;
\draw  [dash pattern={on 4.5pt off 4.5pt}]  (259.05,221.29) -- (258.4,326.76) ;
\draw [shift={(258.38,328.76)}, rotate = 270.36] [color={rgb, 255:red, 0; green, 0; blue, 0 }  ][line width=0.75]    (10.93,-3.29) .. controls (6.95,-1.4) and (3.31,-0.3) .. (0,0) .. controls (3.31,0.3) and (6.95,1.4) .. (10.93,3.29)   ;
\draw  [dash pattern={on 4.5pt off 4.5pt}]  (462.05,221.79) -- (462.54,325.79) ;
\draw [shift={(462.55,327.79)}, rotate = 269.73] [color={rgb, 255:red, 0; green, 0; blue, 0 }  ][line width=0.75]    (10.93,-3.29) .. controls (6.95,-1.4) and (3.31,-0.3) .. (0,0) .. controls (3.31,0.3) and (6.95,1.4) .. (10.93,3.29)   ;
\draw  [dash pattern={on 4.5pt off 4.5pt}]  (434.69,218.52) -- (435.81,324.5) ;
\draw [shift={(435.83,326.5)}, rotate = 269.39] [color={rgb, 255:red, 0; green, 0; blue, 0 }  ][line width=0.75]    (10.93,-3.29) .. controls (6.95,-1.4) and (3.31,-0.3) .. (0,0) .. controls (3.31,0.3) and (6.95,1.4) .. (10.93,3.29)   ;
\draw  [dash pattern={on 4.5pt off 4.5pt}]  (491.52,218.55) -- (492.65,324.53) ;
\draw [shift={(492.67,326.53)}, rotate = 269.39] [color={rgb, 255:red, 0; green, 0; blue, 0 }  ][line width=0.75]    (10.93,-3.29) .. controls (6.95,-1.4) and (3.31,-0.3) .. (0,0) .. controls (3.31,0.3) and (6.95,1.4) .. (10.93,3.29)   ;
\draw  [dash pattern={on 4.5pt off 4.5pt}]  (333.52,222.02) -- (334.65,328) ;
\draw [shift={(334.67,330)}, rotate = 269.39] [color={rgb, 255:red, 0; green, 0; blue, 0 }  ][line width=0.75]    (10.93,-3.29) .. controls (6.95,-1.4) and (3.31,-0.3) .. (0,0) .. controls (3.31,0.3) and (6.95,1.4) .. (10.93,3.29)   ;

\draw (280,10) node [anchor=north west][inner sep=0.75pt]   [align=left] {$\F_\PP(\OO)$};
\draw (90,120) node [anchor=north west][inner sep=0.75pt]   [align=left] {$\F_{\PP_1}(\OO)$};
\draw (200,120) node [anchor=north west][inner sep=0.75pt]   [align=left] {$\F_{\PP_2}(\OO)$};
\draw (290,120) node [anchor=north west][inner sep=0.75pt]   [align=left] {$\F_{\PP_j}(\OO)$};
\draw (390,120) node [anchor=north west][inner sep=0.75pt]   [align=left] {$\F_{\PP_\ell}(\OO)$};
\draw (500,115) node [anchor=north west][inner sep=0.75pt]   [align=left] {$\Delta_{M} =\sum_{j=1}^\ell \Tau_j$};
\draw (7,325) node [anchor=north west][inner sep=0.75pt]   [align=left] {$\kappa=0$};
\draw (6.5,209.5) node [anchor=north west][inner sep=0.75pt]   [align=left] {$\kappa+2$};
\draw (6.5,122.5) node [anchor=north west][inner sep=0.75pt]   [align=left] {$\kappa+1$};
\draw (7.5,26.5) node [anchor=north west][inner sep=0.75pt]   [align=left] {$\kappa=\frac{\ell-n}2<0$};
\draw (500,202.17) node [anchor=north west][inner sep=0.75pt]   [align=left]{$\Delta^2_{M} =\sum_{i,j=1}^\ell \Tau_{i}\Tau_{j}$};
\draw (550,320) node [anchor=north west][inner sep=0.75pt]   [align=left] {$\Delta^{\frac{n-\ell}2}_{M}$};
\draw (198.5,64) node [anchor=north west][inner sep=0.75pt]   [align=left] {$\Tau_1$};
\draw (257,77) node [anchor=north west][inner sep=0.75pt]   [align=left] {$\Tau_2$};
\draw (330,77) node [anchor=north west][inner sep=0.75pt]   [align=left] {$\Tau_j$};
\draw (385,61.5) node [anchor=north west][inner sep=0.75pt]   [align=left] {$\Tau_\ell$};
\draw (85,164.5) node [anchor=north west][inner sep=0.75pt]   [align=left] {$\Tau_{1,1}$};
\draw (482.5,162) node [anchor=north west][inner sep=0.75pt]   [align=left] {$\Tau_{\ell,\ell+2}$};
\draw (20,257.17) node [anchor=north west][inner sep=0.75pt]   [align=left] {$\displaystyle \vdots $};
\draw (563.5,256.67) node [anchor=north west][inner sep=0.75pt]   [align=left] {$\displaystyle \vdots $};

\draw (60,330) node [anchor=north west][inner sep=0.75pt]   [align=left] {$\M(\OO)$};
\draw (120,337) node [anchor=north west][inner sep=0.75pt]   [align=left] {$\cdots$};
\draw (160,330) node [anchor=north west][inner sep=0.75pt]   [align=left] {$\M(\OO)$};
\draw (230,337) node [anchor=north west][inner sep=0.75pt]   [align=left] {$\cdots\cdots\cdots$};
\draw (330,337) node [anchor=north west][inner sep=0.75pt]   [align=left] {$\cdots\cdots\cdots$};
\draw (430,337) node [anchor=north west][inner sep=0.75pt]   [align=left] {$\cdots\cdots$};

\draw (475,330) node [anchor=north west][inner sep=0.75pt]   [align=left] {$\M(\OO)$};

\end{tikzpicture}
\end{center}

\caption{The \emph{Fueter tree} of $\F_\PP(\OO)$. 
The tree has root $\F_\PP(\OO)$, nodes some spaces $\F_{\PP'}(\OO)$ and leaves $\M(\OO)$.  
}
\label{figureTree}
\end{figure}

\begin{remark}\label{rem:qn}
The number of distinct Fueter trees on $M$ (up to equivalence) is equal to the number of non-equivalent roots $\F_\PP(\OO)$, with $\PP$ an odd partition. This is the number of \emph{partitions in odd parts} of $n$, usually denoted by $q(n)$. Excluding the trivial tree with root and leaf $\M(\OO)$, there are, up to equivalence, $q(n)-1$ distinct Fueter trees. Therefore there are $q(n)-1$ Fueter Theorems on $M$, where $n=\dim M-1$. The function $q(n)$ is tabulated in Table \ref{table} for small values of $n$.
\end{remark}

\subsection{Examples}\label{sec:Examples}

\subsubsection{Quaternions}
Since $q(3)=2$ for $M=\hh$, there is only one non-trivial Fueter tree, corresponding to the slice-regular  version of the classical Fueter Theorem on $\hh$ \cite{Fueter1934}:
\[
\Tau=\Delta_4:\sr(\OO)=\F_{\{1,2,3\}}(\OO)\to \ker\dibar_\hh=\F_{\{1\},\{2\},\{3\}}(\OO),
\]
with image $\Tau(\sr(\OO))\subset\ker\dibar_\hh\cap\SL(\OO)$ contained in the subspace of monogenic slice functions, also called axially monogenic quaternionic functions.

\subsubsection{Clifford algebras $\rr_n$}
Let $\OO\subseteq M=\R^{n+1}\subset\R_n$. If $n$ is odd, there is a unary Fueter tree with root $\mathcal{SM}(\OO)$, containing $(n-1)/2$ other unary Fueter trees (see Figure \ref{figureRn}). It correspond to the Fueter-Sce Theorem \cite{Sce}, stating that $\Delta_{M}^{\frac{n-1}2}f$ is axially monogenic for every slice-monogenic function $f$ on $\OO$. As intermediate nodes of this Fueter tree we can take the spaces $\F_{\{1\},\ldots,\{p\},\{p+1,\ldots,n\}}(\OO)$, with $p$ even. They are spaces of generalized partial-slice monogenic functions of type $(p,n-p)$. This gives a meaning to all the iterated euclidean Laplacians $\Delta^i_{M}f$ for a slice monogenic $f$: 
\[
\Delta^i_{M}f\in\F_{\{1\},\ldots,\{2i\},\{2i+1,\ldots,n\}}(\OO)\cap\ker\mathscr S_\PP\text{\quad for $i=1,\ldots,\tfrac{n-1}2$.} 
\]
The subtree with root $\F_{\{1\},\ldots,\{2i\},\{2i+1,\ldots,n\}}(\OO)$ and height $(n-2i-1)/2$ corresponds to the Fueter-Sce Theorem for generalized partial-slice monogenic functions of type $(2i,n-2i)$ recently proved in \cite[Theorem 3.8]{XuSabadiniFueterSce}.

As observed in Remark \ref{rem:Pj}, the first node in the tree with root $\sr(\OO)$ can be replaced by any space $\F_{\PP'}(\OO)$ with $\PP'=\{\{i\},\{j\},[n]\setminus\{i,j\}\}$, for any distinct $i,j\in[n]$, and similarly for the nodes below. Indeed, for any $f\in\sr(\OO)$, 
\[
\delta_2^{\ell} f\in\F_{\{1\},\{2\},\{3,\ldots,n\}}(\OO)\cap\ker\mathscr S_\PP\cap\F_{\{i\},\{j\},[n]\setminus\{i,j\}}(\OO)
\]
for any $\ell\in[n]$.

If $n\le5$, these are all (up to equivalence) the Fueter trees for the paravector space $\R^{n+1}$ of $\R_n$, since $q(3)=2$ and $q(5)=3$. When $n\ge7$, there are other Fueter trees. For example, for $n=7$ we have $q(7)=5$ while $(n+1)/2=4$. The missing tree is a binary tree of height 2 with root the space $\F_{\{1\},\{2,3,4\},\{5,6,7\}}(\OO)$, showing that $\Delta^2_M f\in\M(\OO)\cap\ker\mathscr S_{\{2,3,4\}}\cap\ker\mathscr S_{\{5,6,7\}}$ for every $f\in\F_{\{1\},\{2,3,4\},\{5,6,7\}}(\OO)$, i.e., $\Delta^2_M f$ is a \emph{biaxial monogenic function} on $\OO$. 


 \begin{figure}
     \centering
        
\tikzset{every picture/.style={line width=0.75pt}} 

\begin{tikzpicture}[x=0.75pt,y=0.75pt,yscale=-1,xscale=1]

\draw    (321.74,31.24) -- (321.74,78.24) ;
\draw [shift={(321.74,80.24)}, rotate = 270] [color={rgb, 255:red, 0; green, 0; blue, 0 }  ][line width=0.75]    (10.93,-3.29) .. controls (6.95,-1.4) and (3.31,-0.3) .. (0,0) .. controls (3.31,0.3) and (6.95,1.4) .. (10.93,3.29)   ;
\draw    (321.74,80.24) -- (321.74,127.24) ;
\draw [shift={(321.74,129.24)}, rotate = 270] [color={rgb, 255:red, 0; green, 0; blue, 0 }  ][line width=0.75]    (10.93,-3.29) .. controls (6.95,-1.4) and (3.31,-0.3) .. (0,0) .. controls (3.31,0.3) and (6.95,1.4) .. (10.93,3.29)   ;
\draw  [dash pattern={on 0.84pt off 3.55pt}]  (321.74,129.24) -- (321.74,176.24) ;
\draw  [dash pattern={on 0.84pt off 3.55pt}]  (240,140) -- (240,160) ;
\draw  [dash pattern={on 0.84pt off 3.55pt}]  (355,130) -- (355,170) ;
\draw  [dash pattern={on 0.84pt off 3.55pt}]  (450,140) -- (450,160) ;
\draw [shift={(321.74,178.24)}, rotate = 270] [color={rgb, 255:red, 0; green, 0; blue, 0 }  ][line width=0.75]    (10.93,-3.29) .. controls (6.95,-1.4) and (3.31,-0.3) .. (0,0) .. controls (3.31,0.3) and (6.95,1.4) .. (10.93,3.29)   ;
\draw    (321.74,178.24) -- (321.74,225.24) ;
\draw [shift={(321.74,227.24)}, rotate = 270] [color={rgb, 255:red, 0; green, 0; blue, 0 }  ][line width=0.75]    (10.93,-3.29) .. controls (6.95,-1.4) and (3.31,-0.3) .. (0,0) .. controls (3.31,0.3) and (6.95,1.4) .. (10.93,3.29)   ;

\draw (300,15) node [anchor=north west][inner sep=0.75pt]   [align=left] {$\mathcal{SM}(\OO)$};
\draw (350,40) node [anchor=north west][inner sep=0.75pt]   [align=left] {$\Delta$};
\draw (300,232) node [anchor=north west][inner sep=0.75pt]   [align=left] {$\mathcal{M}(\OO)$};
\draw (350,90) node [anchor=north west][inner sep=0.75pt]   [align=left] {$\Delta$};
\draw (450,90) node [anchor=north west][inner sep=0.75pt]   [align=left] {$\Delta^2$};
\draw (190,65) node [anchor=north west][inner sep=0.75pt]   [align=left] {$\F_{\{1\},\{2\},\{3,\ldots,n\}}(\OO)$};
\draw (380,65) node [anchor=north west][inner sep=0.75pt]   [align=left] {$\mathcal{GSM}(\OO)$\text{ of type $(2,n-2)$}};
\draw (175,115) node [anchor=north west][inner sep=0.75pt]   [align=left] {$\F_{\{1\},\ldots,\{4\},\{5,\ldots,n\}}(\OO)$};
\draw (380,115) node [anchor=north west][inner sep=0.75pt]   [align=left] {$\mathcal{GSM}(\OO)$\text{ of type $(4,n-4)$}};
\draw (140,165) node [anchor=north west][inner sep=0.75pt]   [align=left] {$\F_{\{1\},\{2\},\ldots,\{n-3\},\{n-2,\ldots,n\}}(\OO)$};
\draw (380,165) node [anchor=north west][inner sep=0.75pt]   [align=left] {$\mathcal{GSM}(\OO)$\text{ of type $(n-3,3)$}};
\draw (350,190) node [anchor=north west][inner sep=0.75pt]   [align=left] {$\Delta$};
\draw (450,210) node [anchor=north west][inner sep=0.75pt]   [align=left] {$\Delta^{\frac{n-1}2}$};

\end{tikzpicture}

\caption{The \emph{Fueter tree} for Clifford slice-monogenic functions is a unary tree of height $|\kappa|=(n-1)/2$.
The tree has root $\mathcal{SM}(\OO)$, internal nodes some spaces $\F_{\PP}(\OO)$ (generalized partial-slice monogenic
functions of type $(2i,n-2i)$, $i=1,\ldots,(n-3)/2$) and one leaf $\M(\OO)$. 
\label{figureRn}
}

\end{figure}
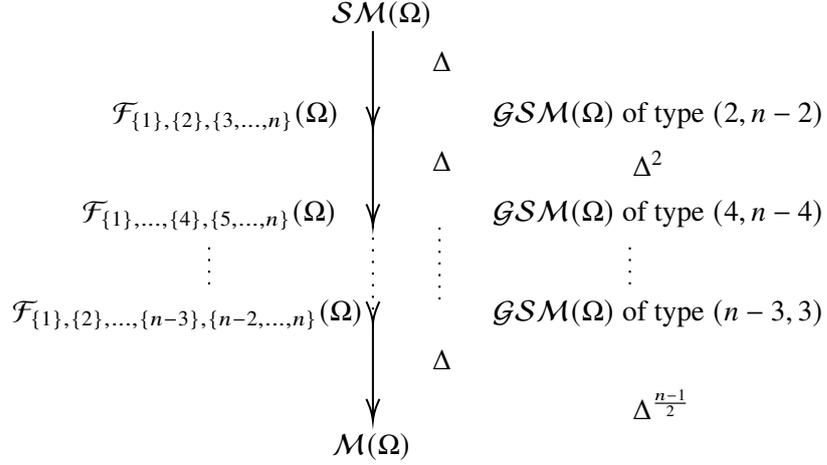

\subsubsection{Octonions}\label{sec:octonions}

There are $q(7)=5$ distinct (up to equivalence) Fueter trees on $\oo$, four of them are subtrees of the unary tree  (see Figure \ref{figureO1}) with root $\sr(\OO)$, and one is the binary tree of height 2 (see Figure \ref{figureO2}) with root $\F_\PP(\OO)$ with $\PP=\{\{1,2,3\},\{4\},\{5,6,7\}\}$ (or equivalent partitions of the set $[7]$, giving the same partition $3+1+3$ of $7$). 

This root space $\F_\PP(\OO)$ is linked to the decomposition $\oo=\hh+\ell\hh$ obtained through the Cayley-Dickson process. Let $\B=(1,i,j,k,\ell,\ell i,\ell j,\ell k)$ be the hypercomplex basis giving $\oo=\hh+\ell\hh$. The Dunkl-Cauchy-Riemann operator with multiplicities ${\bf k}=-\frac13(1,1,1,0,1,1,1)$ is
\[
D_\PP=\dibar_\oo-\tfrac13\sum_{i=1, i\ne4}^7\frac{v_i}{x_i}(1-r_i).
\]
The $\R$-linear functions $x_{\{1,2,3\}}=x_0+x_1i+x_2j+x_3k$, $x_{\{5,6,7\}}=x_0+x_5(\ell i)+x_6(\ell j)+x_7(\ell k)$ and $x_{\{4\}}=x_0+x_4\ell$ and their powers belong to $\F_\PP(\oo)$. It follows that the space $\F_{\PP}(\OO)$ contains two copies of the space of quaternionic slice-regular functions and a copy of the space of $\cc_\ell$-holomorphic functions. 

Since $q(7)-1=4$, we  have four distinct octonionic Fueter Theorems, corresponding to the maps:
\begin{itemize}
  \item[(i)] $\Delta^3:\sr(\OO)\to\M(\OO)$, with image in $\AM(\OO)$ (axially monogenic functions). This is the generalization of Dentoni-Sce Theorem \cite{DentoniSce} to slice-regular functions \cite[Theorem 27]{CRoperators}; 
  \item[(ii)] $\Delta^2:\F_{\{1\},\{2\},\{3,4,5,6,7\}}(\OO)\simeq\mathcal{GSM}(\OO)\text{ of type $(2,5)$}\to\M(\OO)$, with image in $\M(\OO)\cap\ker\mathscr S_{\{3,4,5,6,7\}}$;
  \item[(iii)] $\Delta:\F_{\{1\},\{2\},\{3\},\{4\},\{5,6,7\}}(\OO)\simeq\mathcal{GSM}(\OO)\text{ of type $(4,3)$}\to\M(\OO)$, with image in $\M(\OO)\cap\ker\mathscr S_{\{5,6,7\}}$;
  \item[(iv)] $\Delta^2:\F_{\{1,2,3\},\{4\},\{5,6,7\}}(\OO)\to\M(\OO)$, with image in $\M(\OO)\cap\ker\mathscr S_{\{1,2,3\}}\cap\ker\mathscr S_{\{5,6,7\}}$ (biaxial monogenic).
\end{itemize}
Here $\Delta=\Delta_8$ is the Euclidean Laplacian of $\rr^8$. 

As an example for the last Fueter Theorem (case (iv) of the above list), take the CK-extension $f=P_{(2,1,2)}\in\F_{\{1,2,3\},\{4\},\{5,6,7\}}(\oo)$ of the polynomial function $g=\x_{\{1,2,3\}}^2\x_4\x_{\{5,6,7\}}^2$. 
A direct computation shows that 
\begin{align*}
T_1(f)&=-\delta_2^1 f=-\tfrac43 x_0^3+4x_0\big(|\IM(x'')|^2-2x_4\IM(x'')\big)+\ell\big(4x_4(|\IM(x'')|^2-x_0^2)\big),
\\
T_3(f)&=-\delta_2^5 f=-\tfrac43 x_0^3+4x_0|\IM(x')|^2+\ell\big(4x_4(|\IM(x')|^2+2x_0\IM(x')-x_0^2\big),
\end{align*}
where $x=x'+\ell x''\in\hh+\ell\hh=\oo$. Since $T_2\equiv0$, Theorem \ref{teo:DeltaM} gives
\begin{align*}
\Delta f&=T_1(f)+T_3(f)=-\tfrac83 x_0^3+4x_0\big(|\IM(x')|^2+|\IM(x'')|^2-2x_4\IM(x'')\big)\\
&\quad + \ell\big(4x_4(|\IM(x')|^2+|\IM(x'')|^2+2x_0\IM(x')-2x_0^2)).
\end{align*}
The double Laplacian $\Delta^2 f=-\delta_2^5(T_1(f))-\delta_2^1(T_3(f))$ is the monogenic function $32 x_{\{4\}}=32(x_0+ x_4\ell)$. 

\begin{figure}
  
     \begin{minipage}{0.50\textwidth}
         \centering

\tikzset{every picture/.style={line width=0.75pt}} 

\begin{tikzpicture}[x=0.75pt,y=0.75pt,yscale=-1,xscale=1]

\draw    (334.01,50.44) -- (334.34,116.62) ;
\draw [shift={(334.35,118.62)}, rotate = 269.71] [color={rgb, 255:red, 0; green, 0; blue, 0 }  ][line width=0.75]    (10.93,-3.29) .. controls (6.95,-1.4) and (3.31,-0.3) .. (0,0) .. controls (3.31,0.3) and (6.95,1.4) .. (10.93,3.29)   ;
\draw    (334.69,186.8) -- (335.02,252.98) ;
\draw [shift={(335.03,254.98)}, rotate = 269.71] [color={rgb, 255:red, 0; green, 0; blue, 0 }  ][line width=0.75]    (10.93,-3.29) .. controls (6.95,-1.4) and (3.31,-0.3) .. (0,0) .. controls (3.31,0.3) and (6.95,1.4) .. (10.93,3.29)   ;
\draw    (334.35,118.62) -- (334.68,184.8) ;
\draw [shift={(334.69,186.8)}, rotate = 269.71] [color={rgb, 255:red, 0; green, 0; blue, 0 }  ][line width=0.75]    (10.93,-3.29) .. controls (6.95,-1.4) and (3.31,-0.3) .. (0,0) .. controls (3.31,0.3) and (6.95,1.4) .. (10.93,3.29)   ;

\draw (310,32) node [anchor=north west][inner sep=0.75pt]   [align=left] {$\sr(\OO)$};
\draw (250,100) node [anchor=north west][inner sep=0.75pt]   [align=left] {$\mathcal{GSM}(\OO)$\\\text{of type $(2,5)$}};
\draw (360,100) node [anchor=north west][inner sep=0.75pt]   [align=left] {$\Delta(\sr(\OO))$\\Slice Fueter};
\draw (315,260) node [anchor=north west][inner sep=0.75pt]   [align=left] {$\M(\OO)$};
\draw (250,172) node [anchor=north west][inner sep=0.75pt]   [align=left] {$\mathcal{GSM}(\OO)$\\\text{of type $(4,3)$}};
\draw (352,210) node [anchor=north west][inner sep=0.75pt]   [align=left] {$\Delta$};
\draw (352,140) node [anchor=north west][inner sep=0.75pt]   [align=left] {$\Delta$};
\draw (352,70) node [anchor=north west][inner sep=0.75pt]   [align=left] {$\Delta$};

\end{tikzpicture}

\caption{{The \emph{Fueter tree} for octonionic slice-regular functions is a unary tree of height 3.
The tree has root $\sr(\OO)$, one leaf $\M(\OO)$ and as first node one can take the space of Slice Fueter-regular functions.}
\label{figureO1}
}

\end{minipage}\hfill
     \begin{minipage}{0.5\textwidth}
         \centering

\tikzset{every picture/.style={line width=0.75pt}} 

\begin{tikzpicture}[x=0.75pt,y=0.75pt,yscale=-1,xscale=1]

\draw    (351.09,28.79) -- (272.57,99.45) ;
\draw [shift={(271.09,100.79)}, rotate = 318.01] [color={rgb, 255:red, 0; green, 0; blue, 0 }  ][line width=0.75]    (10.93,-3.29) .. controls (6.95,-1.4) and (3.31,-0.3) .. (0,0) .. controls (3.31,0.3) and (6.95,1.4) .. (10.93,3.29)   ;
\draw    (351.09,28.79) -- (430.27,100.12) ;
\draw [shift={(431.75,101.46)}, rotate = 222.01] [color={rgb, 255:red, 0; green, 0; blue, 0 }  ][line width=0.75]    (10.93,-3.29) .. controls (6.95,-1.4) and (3.31,-0.3) .. (0,0) .. controls (3.31,0.3) and (6.95,1.4) .. (10.93,3.29)   ;
\draw    (431.75,101.46) -- (432.41,186.79) ;
\draw [shift={(432.42,188.79)}, rotate = 269.56] [color={rgb, 255:red, 0; green, 0; blue, 0 }  ][line width=0.75]    (10.93,-3.29) .. controls (6.95,-1.4) and (3.31,-0.3) .. (0,0) .. controls (3.31,0.3) and (6.95,1.4) .. (10.93,3.29)   ;
\draw    (271.09,100.79) -- (271.74,186.13) ;
\draw [shift={(271.75,188.13)}, rotate = 269.56] [color={rgb, 255:red, 0; green, 0; blue, 0 }  ][line width=0.75]    (10.93,-3.29) .. controls (6.95,-1.4) and (3.31,-0.3) .. (0,0) .. controls (3.31,0.3) and (6.95,1.4) .. (10.93,3.29)   ;

\draw (290,5) node [anchor=north west][inner sep=0.75pt]   [align=left] {$\F_{\{1,2,3\},\{4\},\{5,6,7\}}(\OO)$};
\draw (293.33,45) node [anchor=north west][inner sep=0.75pt]   [align=left] {$T_1$};
\draw (396.67,45) node [anchor=north west][inner sep=0.75pt]   [align=left] {$T_3$};
\draw (440,45) node [anchor=north west][inner sep=0.75pt]   [align=left] {$\Delta=T_1+T_3$};
\draw (325,60) node [anchor=north west][inner sep=0.75pt]   [align=left] {($T_2\equiv0$)};
\draw (435,89) node [anchor=north west][inner sep=0.75pt]   [align=left] {$\F_{\PP_3}(\OO)$};
\draw (220,91) node [anchor=north west][inner sep=0.75pt]   [align=left] {$\F_{\PP_1}(\OO)$};
\draw (235.33,131) node [anchor=north west][inner sep=0.75pt]   [align=left] {$\Delta$};
\draw (447.33,131.67) node [anchor=north west][inner sep=0.75pt]   [align=left] {$\Delta$};
\draw (410,195) node [anchor=north west][inner sep=0.75pt]   [align=left] {$\M(\OO)$};
\draw (300,180) node [anchor=north west][inner sep=0.75pt]   [align=left] {$\Delta^2=T_5T_1+T_1T_3$};
\draw (250,195) node [anchor=north west][inner sep=0.75pt]   [align=left] {$\M(\OO)$};

\end{tikzpicture}

\caption{{Binary \emph{Fueter tree} for octonionic Dunkl-regular functions with root $\F_\PP(\OO)$, where $\PP=\{\{1,2,3\},\{4\},\{5,6,7\}\}$. There are two nodes $\F_{\PP_1}(\OO)$ and $\F_{\PP_1}(\OO)$ with $\PP_1=\{\{1\},\ldots,\{4\},\{5,6,7\}\}$}, $\PP_3=\{\{1,2,3\},\{4\},\ldots,\{7\}\}$ and two leaves $\M(\OO)$.
\label{figureO2}
}

   \end{minipage}
\end{figure}

\section*{Aknowledgments}




\end{document}